\documentclass{amsart}

\usepackage[latin1,utf8]{inputenc}
\usepackage{subfigure}
\usepackage{amssymb,amsmath,amsthm,stmaryrd,comment,enumitem}
\usepackage{bm}

\usepackage{graphicx}
\usepackage{hyperref}
\usepackage{epsfig}
\usepackage{color}

\usepackage{tikz}

\def\AiC{\mathrm{AiC}}
\def\Ai{\mathrm{Ai}}
\def\UT{\mathrm{UT}}

\def\cdec{C-decorated\ }
\def\cE{\mathcal{E}}
\def\cT{\mathcal{T}}
\def\cA{\mathcal{A}}
\def\cB{\mathcal{B}}
\def\cC{\mathcal{C}}

\def\eps{\epsilon}
\def\cCgc{\cC_g^{\circ}}
\def\cP{\mathcal{P}}
\def\cPv{\mathcal{P}^{\mathrm{v}}}
\def\cPc{\mathcal{P}^{\mathrm{c}}}
\def\cTgv{\cT_g^{\mathrm{v}}}
\def\Cat{\mathrm{Cat}}
\def\Nar{\mathrm{Nar}}

\newtheorem{theo}{Theorem}
\newtheorem{lem}[theo]{Lemma}
\newtheorem{prop}[theo]{Proposition}
\newtheorem{corol}[theo]{Corollary}

\newtheorem{defin}{Definition}

\theoremstyle{remark}
\newtheorem{remark}{Remark}

\def\NN{\mathbb{N}}
\def\pp{\bm{p}}
\def\qq{\bm{q}}
\def\wt{\mathrm{wt}}
\renewcommand\phi{\varphi}
\def\Bi{\mathrm{Bi}}
\def\BiL{\mathrm{BiL}}
\def\II{\bm{I}}
\def\JJ{\bm{J}}
\def\HH{\bm{H}}
\def\pp{\bm{p}}
\def\qq{\bm{q}}
\def\xx{\bm{x}}
\def\yy{\bm{y}}
\def\BiC{\mathrm{BiC}}

\def\KK{\bm{K}}
\def\LL{\bm{L}}
\def\MM{\bm{M}}
\def\aa{\bm{a}}
\def\bb{\bm{b}}
\def\cc{\bm{c}}
\def\CCC{\mathcal{C}}

\author[G. Chapuy, V. Féray and É. Fusy]{
Guillaume Chapuy,
Valentin F\'eray
and \'Eric Fusy}
\title{%
A simple model of trees for unicellular maps
}
\address[G.C.]{CNRS, LIAFA - UMR 7089, Universit\'e Paris 7, 75205 Paris Cedex, France\\
guillaume.chapuy@liafa.univ-paris-diderot.fr}
\address[V.F.]{CNRS, LaBRI - UMR 5800, Universit\'e de Bordeaux, 33405 Talence Cedex, France\\
feray@labri.fr}
\address[É.F.]{CNRS, LIX - UMR 7161, \'Ecole Polytechnique, 91128 Palaiseau Cedex, France\\
fusy@lix.polytechnique.fr}
\thanks{First and third author partially supported by the ERC grant StG
208471 -- ExploreMaps.}
\thanks{Second author partially support by ANR project PSYCO}


\begin{document}
\begin{abstract}
We consider unicellular maps, or polygon gluings, of fixed genus.
%Several formulas have been obtained in the last forty years for different
%counting problems related to these objects, a few  of them having bijective proofs.
%GC7: removed ref. to FPSAC:
A few years ago 
the first author gave a recursive bijection transforming
unicellular maps into trees, explaining the presence of Catalan numbers
in counting formulas for these objects.
In this
paper, we give another bijection that explicitly describes the
``recursive part'' of the first bijection.
As a result we obtain a very simple description
of unicellular maps as pairs made by a plane tree and a permutation-like
structure.

 All the previously known formulas follow as an immediate corollary or
easy exercise, thus giving a bijective proof for each of them, in a unified way.
For some of these formulas, this is the first bijective proof, 
{\it e.g.} the Harer-Zagier recurrence formula,
the Lehman-Walsh formula and the Goupil-Schaeffer formula.
We also discuss several applications of our construction:
we obtain a new proof of an identity related to covered maps
due to Bernardi and the first author, and 
thanks to previous work of the second author,
% GC: Valentin, t'aurais pas présenté ton résultat à un FPSAC qu'on puisse
% décommenter la parenthèse suivante:
%(FPSAC '08) 
% VF: Non malheureusement, c'était un autre papier à ce FPSAC :(
%GC3 : singulier
we give a new expression for Stanley character polynomials,
which evaluate irreducible characters of the symmetric group. Finally, we
show that our techniques apply partially to unicellular 3-constellations and
to related objects that we call quasi-3-constellations.
%
%\paragraph{R\'esum\'e.}
%Nous consid\'erons des cartes orient\'ees \`a une face de genre fix\'e.
%%GC3
%\`A SFCA'09 le premier auteur a introduit une bijection r\'ecursive
%envoyant une carte unicellulaire vers un arbre, ce qui permet
%d'obtenir des formules \'enum\'eratives pour les cartes \`a une face
%(et en particulier la pr\'esence des nombres de Catalan).
%Dans l'article ici pr\'esent, et en nous appuyant sur la bijection ci-dessus, 
%nous obtenons une incarnation tr\`es simple
%des cartes \`a une face comme des paires form\'ees d'un arbre plan et d'une 
%permutation d'un certain type. Toutes les formules pr\'ec\'edemment connues
%d\'ecoulent ais\'ement de cette nouvelle incarnation, donnant des preuves
%bijectives dans un cadre unifi\'e. Pour certaines de ces formules, telles que
% la r\'ecurrence de Harer-Zagier
%ou les formules de Lehman-Walsh/Goupil-Schaeffer, nous obtenons la premi\`ere
%preuve bijective connue. Par ailleurs, en combinant notre approche avec des
% travaux du second auteur, nous obtenons une nouvelle expression
%pour les polynômes de Stanley qui donnent certaines \'evaluations des
%caract\`eres du groupe sym\'etrique. 
%
%
\end{abstract}
\maketitle

%GC: tentative d'intro. Je vais droit au but, puis donne quelques détails.
% J'aimerais faire plus long mais pas trop la place.
\section{Introduction}
% GC: rajouter quelques références là-dedans, mais ne pas tout mettre au risque
% de finir avec une bilio de 3 pages (s'en tenir à ce qu'on cite déjà plus
% tard?)
% VF: j'ai mis comme référence générale sur le sujet et ses connexions avec
% le projet Appollo 13 le bouquin de Zvonkine\dots
% ça a le mérite de ne pas allonger la biblio et de décourager tout lecteur
% de chercher à en savoir plus :), mais je ne sais pas si ça vous va.
A unicellular map is a connected graph embedded in a surface in such a
way that the complement of the graph is a topological disk.
These objects have appeared frequently in combinatorics in the last forty
years, in relation with the general theory of map enumeration, but also with
the representation theory of the symmetric group, the study of
%VF10: added ``the computation of matrix integrals''
permutation factorizations, the computation of matrix integrals
or the study of moduli spaces of curves.
%Many formulas have been obtained for different counting problems related
%to these objects, with different techniques.
%, and a few of them have been given bijective proofs
All these connections have turned the enumeration of unicellular maps
into an important research field (for the many connections with other areas,
see \cite{ZvLa97} and references therein; for an overview of the results see
the introductions of the papers \cite{Ch09,BeHa})
%GC2013: je n'ai pas voulu ajouter la phrase suivante pourtant demandée
%explicitement par l'un des referees:
% Note however that the present paper answers the open questions left in \cite{Ch09} 
The counting formulas for unicellular maps that appear in the literature 
can be roughly separated into two types.
% EF9 types plutot que families. 
% GC2: je mets mon papier et celui d'Olivier comme références générales
% supplémentaires, ca sur les cartes à une face c'est tout de même plus fourni

The first type deals with {\em colored} maps
(maps endowed with a %EF9 mapping plutot que application
mapping from its vertex set to a set of $q$ colors). 
This implies ``summation'' enumeration formulas 
(see \cite{HaZa86,%Ja88,
%VF: I removed this reference because I didn't find the
%``Jackson summation formula'' in this article (even under a different form)
%Perhaps I have missed it, or the reference is not good, or Jackson never wrote
%such formulas.
ScVa08,MoVa09} or paragraph \ref{SubsectColoredMaps} below).
These formulas are often elegant, and different combinatorial proofs for
them have been given in the past few years \cite{La01,GoNi,ScVa08,MoVa09,BeHa}.
The issue is that some important topological information, such as the genus of the surface,
is 
% GC2: ``lost'' n'est-il pas un peu abusif ?
%lost
not apparent
in these constructions.

Formulas of the second type keep track 
%GC2: added
explicitly
of the genus of the surface; they are either inductive relations, like the Harer-Zagier recurrence
formula~\cite{HaZa86}, 
%EF enleve cumbersome qui est un peu negatif
or are explicit (but quite involved) closed forms, like the Lehman-Walsh~\cite{LeWa72} 
and the Goupil-Schaeffer~\cite{GoSc00} formulas. 
From a combinatorial point of view, these formulas are harder to understand.
%GC3: modifié, Valentin, tu peux en enlever à nouveau si tu as besoin de place
%pour ce que tu vas rajouter ici.
A step in this direction was done by the first author in \cite{Ch09}
(this construction is explained in subsection \ref{SubsectTrisections}),
which led to new induction relations and to new formulas. 
However the link with other formulas of the second type remained mysterious,
and \cite{Ch09} left open the problem of finding combinatorial proofs of these
formulas.
%GC3 je ne suis pas satisfait du style mais je suis content de l'avoir dit...

%GC2: reformulated
%that left open 
%
%the problem of finding
%combinatorial proofs of the formulas in this second family.

The goal of this paper is to present a new bijection between unicellular maps
and surprisingly simple objects which we call \emph{C-decorated} trees
(these are merely plane trees equipped with a certain kind
of permutation on their vertices).
%VF 28/11: j'ai rajouté la phrase suivante, car l'intro ne parlait pas du
%tout de la construction elle-même mais juste du résultat et de ses corollaires
This bijection,
% GC8: ajouté le plan de papier en même temps que l'intro, du genre:
presented in Section~\ref{SectBij},
is based on the previous work of the first author \cite{Ch09}:
we explicitly describe the ``recursive part'' appearing in this work.
As a consequence, not only can we reprove all
the aforementioned formulas in a bijective way, thus giving the first bijective
proof for several of them, but we do that in a unified way.
Indeed, C-decorated trees are so simple combinatorial objects that all
formulas follow from our bijection as an immediate corollary or easy exercise, 
% GC8 added
as we will see in Section~\ref{SectCounting}.

Another interesting application of this bijection,
% GC8 added
studied in Section~\ref{SectStanley},
is a new explicit way
of computing the so-called Stanley character polynomials.
%VF10 : j'ai coupe cette phrase
The latter are nothing but the evaluation of irreducible characters of the symmetric
groups, properly normalized and parametrized. Indeed, in a previous work~\cite{Fe}, the
second author expressed these polynomials as a generating function of (properly
weighted) unicellular maps.
Although we do not obtain a ``closed form'' expression (there is no reason to
believe that such a form exists!), we express Stanley character polynomials as
the result of a term-substitution in free cumulants,
which are another meaningful quantity in
representation theory of symmetric groups.
% VF: j'ai changé la fin de cette phrase, car les cumulants libres n'ont pas
% d'expression explicite, comme l'annonçait Guillaume.
% GC2: ne pourrait-on pas tout de même dire que tout cela est
% ``computationnel'' ?

% GC: après cette vague intro, un peu plus de contenu
%Let us now describe more precisely our method and results. In genus $0$,
%unicellular maps coincide with plane trees, and are enumerated by Catalan
%numbers. In higher genus, the number of unicellular maps of a given size is the
%product of a polynomial and a genus dependent polynomial. A combinatorial proof
%of that fact was given only recently by the first author~\cite{Ch09}

%\alert{renvoi vers la partie : constellations}
%GC8: voici une proposition
In Section~\ref{SectConstellations} we discuss the possibility of
applying our tools to \emph{$m$-constellations}.
%VF10 : j'ai coupe cette phrase
This notion is a generalization of the notion bipartite maps introduced 
in connection with the study of factorizations in the symmetric group.
%EF9 ai un peu reecrit la phrase
A remarkable formula by Poulalhon and Schaeffer~\cite{PoulalhonSchaefferConstellations} 
(proved with the help of algebraic tools) suggests the possibility of a combinatorial proof using technique similar to ours. 
Although our bijection does not apply to these objects, we present two
partial results in this direction, in the case of $3$-constellations.
One of them 
is an enumeration formula 
for a related family of objects that we call \emph{quasi-$3$-constellations},
that turns out to be  surprisingly similar 
to the Poulalhon-Schaeffer formula.

\section{The main bijection}
\label{SectBij}
\subsection{Unicellular maps and \cdec trees}
%VF10: ajoute un peu de liant
%GC2013: denomination -> terminology
We first briefly review some standard terminology for maps.

A \emph{map} $M$ of genus $g\geq 0$ is a connected graph $G$ embedded on a
closed compact oriented\footnote{Maps can also be defined on non-orientable
surfaces. However, for non-orientable surfaces, only an asymptotic
version~\cite{BernardiChapuyNonOrientable} of the bijection of \cite{Ch09}
has been discovered so far. Since the (complete) bijection of \cite{Ch09} is an essential building block
of all the results presented in the present paper, we will not talk of
non-orientable surfaces in this paper. }
 surface $S$ of genus $g$, such that
%GC: il faut mettre ``oriented'' et non ``orientable'' surface, sinon ça ne
%colle pas vraiment
$S\backslash G$ is a collection of topological disks, which are called the \emph{faces}
of $M$. Loops and multiple edges are allowed. The (multi)graph $G$ is called the \emph{underlying graph} of $M$ and $S$ its
\emph{underlying surface}. Two maps
that differ only by an oriented homeomorphism between the underlying surfaces
are considered the same.
A corner of $M$ is the angular sector between two consecutive
edges around a vertex. A \emph{rooted map} is a map with a marked corner, called
the \emph{root}; 
the vertex incident to the root is called the \emph{root-vertex}. 
%EFMay rajoute cas carte un sommet
By convention, the map with one vertex and no edge (of genus $0$)  
is considered 
as rooted at its unique vertex (the entire sector around the vertex
is considered as a corner, which is the root).%   
From now on, all maps are assumed to be rooted (note that the underlying
graph of a rooted map is naturally vertex-rooted). A \emph{unicellular map} is a map with a unique face. The classical Euler relation $|V|-|E|+|F|=2-2g$ ensures that a unicellular map with $n$ edges has $n+1-2g$ vertices. 
 A \emph{plane tree}
is a unicellular map of genus $0$.

A \emph{rotation system} on a connected graph $G$ consists in a cyclic ordering of the
half-edges of $G$ around each vertex. Given a map $M$, its underlying graph
$G$ is naturally equipped with a rotation system given by the \emph{clockwise
ordering} of half-edges on the surface in a vicinity of each vertex. It is
well-known 
% GC: pas de ref pour gagner 3cm...
that this correspondence is $1$-to-$1$, i.e., a map can be
considered as a connected graph equipped
with a rotation system (thus, as a purely combinatorial object).
We will take this viewpoint from now on. \medskip

%VF10: ajoute un peu de liant (suite)
We now introduce a new object called \cdec tree.

A \emph{cycle-signed} permutation is a permutation where each cycle
carries a sign, either $+$ or $-$. A \emph{C-permutation} is a cycle-signed
permutation where all cycles have odd length, see Figure~\ref{fig:C-tree}(a). 
For each C-permutation $\sigma$ on $n$ 
elements, the \emph{rank} of $\sigma$ is defined as $r(\sigma)=n-\ell(\sigma)$,
where $\ell(\sigma)$ is the number of cycles of $\sigma$. Note that $r(\sigma)$ is even since all cycles have odd
length. The \emph{genus} of $\sigma$ is defined as $r(\sigma)/2$.  
A \emph{\cdec tree} on $n$
edges 
is a pair $\gamma=(T,\sigma)$ where $T$ is a plane tree with $n$ edges and $\sigma$
is a C-permutation of $n+1$ elements. The \emph{genus} of $\gamma$ is defined to be the genus of $\sigma$.  Note that the $n+1$ vertices of $T$
can be canonically numbered from $1$ to $n+1$ ({\it e.g.},
following a left-to-right depth-first traversal), 
hence $\sigma$ can be seen as a permutation of the vertices of $T$, see Figure~\ref{fig:C-tree}(c). 
The \emph{underlying graph} of $\gamma$ is the (vertex-rooted) graph $G$
%EFMay add with n edges 
with $n$ edges 
that is  obtained
from $T$ by merging into a single vertex the vertices in each cycle of $\sigma$ 
(so that the vertices of $G$ correspond to the cycles of $\sigma$), see Figure~\ref{fig:C-tree}(d).  

\begin{figure}
%\begin{flushright}
%  \begin{minipage}{9cm}
  \includegraphics[width=\linewidth]{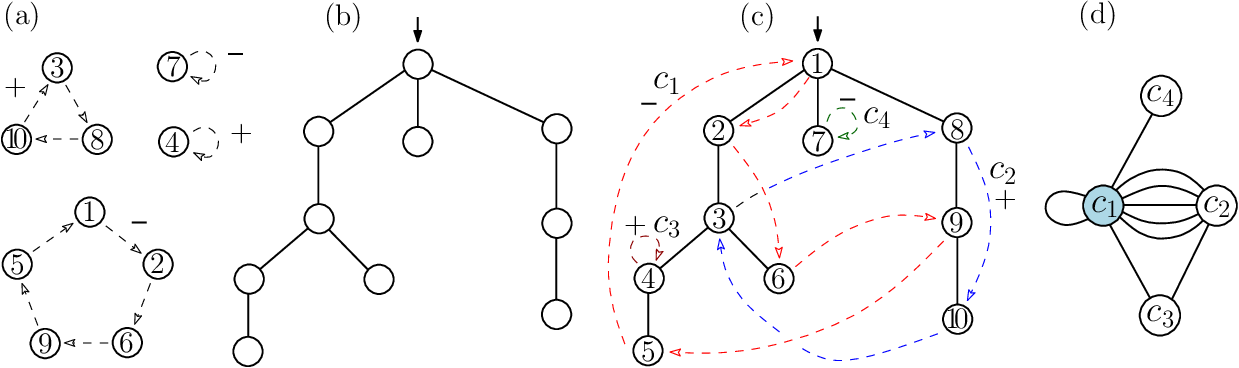}
%\end{minipage}\hfill
%\begin{minipage}{35mm}
%EF6: ai remis les captions des figures en inline
\caption{
(a) \ A C-permutation $\sigma$. (b) \  A plane tree $T$. 
 (c) \ The \cdec tree $(T,\sigma)$. (d) \  The underlying graph of $(T,\sigma)$.
}
\label{fig:C-tree}%\smallskip
%\
%\hspace{1cm}\begin{minipage}{4cm}
%\begin{enumerate}[topsep=0pt,partopsep=0pt, itemsep=0pt, parsep=0pt,
%    leftmargin=15.5pt]
% (a) \ A C-permutation $\sigma$. \\
 % (b) \  A plane tree $T$.
 % \end{minipage} \hfill
 % \begin{minipage}{6cm}
 %(c) \ The \cdec tree $(T,\sigma)$.  \\
% (d) \  The underlying graph of $(T,\sigma)$.
% \end{minipage}\hspace{1cm}
%\end{enumerate}
%\end{flushright}
%GC2: Éric, j'ai sacagé ta figure, mais gagné quelques millimètres... 
% J'ai aussi déplacé les lettre a,b,c,d dans la figure pour gagner de la
%place
%VF5: j'ai remis en page cette figure pour adapter au style amsart
\end{figure}

\begin{defin}
For $n,g$ nonnegative integers, denote by $\cE_g(n)$ the set of unicellular
maps of genus $g$ with $n$ edges; and denote by $\cT_g(n)$ the set
of \cdec trees of genus $g$ with $n$ edges. 
\end{defin}
%VF10: ajoute la def de la notation +
For two finite sets $\cA$ and $\cB$, we denote by $\cA + \cB$ their disjoint
union and by $k\cA$ the set made of $k$ disjoint copies of $\cA$.
Besides, we write $\cA\simeq\cB$ if there is a 
bijection between $\cA$ and $\cB$. Our main result will be to show that
$2^{n+1}\cE_g(n)\simeq \cT_g(n)$,
%VF : je pense qu'il faut préciser ici que ça préserve le graphe sous-jacent.
% Sinon on se demande pourquoi on ne le formule pas de manière énumérative
with a bijection which preserves the underlying graphs of the objects.

\subsection{Recursive decomposition of unicellular maps}
\label{SubsectTrisections}
In this section, we briefly recall a combinatorial method developed in~\cite{Ch09} to decompose
unicellular maps. 
\begin{prop}[Chapuy~\cite{Ch09}]\label{prop:Eg_decomp}
For $k\geq 1$, denote by $\cE_g^{(2k+1)}(n)$ the set of maps from $\cE_g(n)$
in which a set of $2k+1$ vertices is distinguished. Then for $g>0$ and $n\geq 0$,
\begin{align}\label{eq:trisectionrec-maps}
2g\ \!\cE_g(n)\simeq
\cE_{g-1}^{(3)}(n)+\cE_{g-2}^{(5)}(n)+\cE_{g-3}^{(7)}(n)+\cdots+\cE_{0}^{(2g+1)}(n).
\end{align}
In addition, if $M$ and $(M',S')$ are in correspondence,
then the underlying graph of $M$ is obtained from the underlying graph of $M'$
by merging the vertices in $S'$ into a single vertex. 
\end{prop}
We now sketch briefly the construction of \cite{Ch09}. Although this is not really
needed for the sequel, we believe that it gives a good insight into the objects
we are dealing with (readers in a hurry may take
Proposition~\ref{prop:Eg_decomp}
for granted and jump directly to subsection \ref{subsec:apresrappel}). We refer to \cite{Ch09} for proofs and details.

We first explain where the factor $2g$ comes from in
\eqref{eq:trisectionrec-maps}. Let $M$ be a rooted unicellular map of genus $g$
with $n$ edges. Then $M$ has $2n$ corners, and we label them from $1$ to $2n$
incrementally, starting from the root, and going 
% GC2: j'hésite à chaque fois, mais je pense que c'est clockwise ?
% ça dépend si on se voit à l'intérieur ou à l'extérieur de la face...
clockwise around the (unique) face of
$M$ (Figure~\ref{fig:maptrisection}). Let $v$ be a vertex of $M$, let $k$ be
its degree, and let $(a_1,a_2,\dots,a_k)$ be the sequence of the labels of
corners incident to it, read in 
%EFMay il me semble que  c'est plutot clockwise en reagardant la figure
\emph{clockwise direction around $v$} starting
%\emph{counterclockwise direction around $v$} starting
from the minimal label $a_1=\min \{a_i\}$. If for some $j$ lying in 
$\llbracket
1,k-1\rrbracket$ ({\em i.e.} in the set of integers between $1$ and $k-1$,
including $1$ and $k-1$), we have
$a_{j+1}<a_j$, we say that the corner of $v$ labelled by
$a_{j+1}$ is a \emph{trisection of $M$}.
Figure~\ref{fig:maptrisection}(a) shows a map of genus two having four
trisections. More generally we have:
%The  following 
%lemma is the keystone of~\cite{Ch09}:
\begin{lem}[\cite{Ch09}]\label{lem:trisectionscount}
  A unicellular map of genus $g$ contains exactly $2g$ trisections. In other
  words,
  the set of unicellular maps of genus $g$ with $n$ edges and a marked
  trisection is isomorphic to $2g\ \!\cE_g(n)$.
\end{lem}

Now, let $\tau$ be a trisection of $M$ of label $a(\tau)$, and let $v$ be the vertex it belongs to.
We denote by $c$ the corner of $v$ with minimum label
and by $c'$ the corner with minimum label among those which
appear between $c$ and $\tau$ clockwise around $v$ and whose label is greater
than $a(\tau)$. By
definition of a trisection, $c'$ is well defined. We
then construct a new map $M'$, by
\emph{slicing} the vertex $v$ into three new vertices using the three corners
$c,c',\tau$ as in Figure~\ref{fig:maptrisection}(b). We say
that the map $M'$ is obtained from $M$ by \emph{slicing the
trisection $\tau$.}
% GC2: enlevé un saut de ligne, gagne un petit cm
As shown in \cite{Ch09}, the new map $M'$ is a unicellular map of genus $g-1$.
We can thus relabel the $2n$ corners of $M'$ from $1$ to $2n$, according to the
procedure we already used for $M$. Among these corners, three of them, say
$c_1,c_2,c_3$ are naturally inherited from the slicing of $v$, as on
Figure~\ref{fig:maptrisection}(b). Let $v_1, v_2, v_3$ be the vertices they
belong to, respectively. Then the following is true~\cite{Ch09}: {\it In the
map $M'$, the corner $c_i$ has the smallest label around the vertex $v_i$, for
$i\in\{1,2\}$. For $i=3$, either the same is true, or $c_3$ is a trisection of the
map $M'$.}

\begin{figure}
\includegraphics[scale=1]{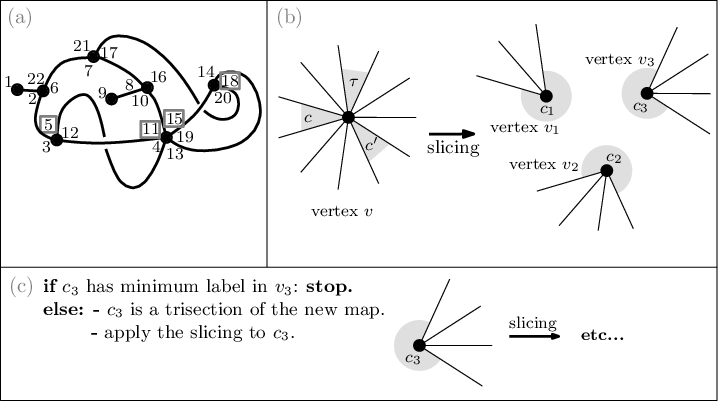}
  \caption{
(a) A unicellular map of genus $2$ equipped with its corner labelling.
Labels corresponding to trisections are boxed.
(b) Given a trisection $\tau$, two other corners of interest $c$ and $c'$ are
canonically defined (see text). ``Slicing the trisection'' then gives rise to three
new vertices $v_1, v_2, v_3$, with distinguished corners $c_1,c_2,c_3$. (c) The recursive
procedure of~\cite{Ch09}: if $c_3$ is the
minimum corner of $v_3$, then stop; else, as shown in \cite{Ch09},  $c_3$ is a
trisection of the new map $M'$: in this case, iterate the slicing operation on
$(M',c_3)$.}\label{fig:maptrisection}
\end{figure}
%VF5: il faut remettre en page cette figure,
%mais je laisse faire l'artiste sur ce coup-là..
%GC7: vive le style amsart, j'y ai passe deux heures...
We now finally describe the bijection promised in 
Proposition~\ref{prop:Eg_decomp}. It is defined recursively on the genus, as
follows. Given a map $M\in \cE_g(n)$ with a marked 
trisection $\tau$, let $M'$ be obtained from $M$ by slicing $\tau$, and
let $c_i,v_i$ be defined as above for $i\in\{1,2,3\}$. If $c_3$ has
the minimum label at $v_3$, set $\Psi(M,\tau):=(M',\{v_1,v_2,v_3\})$, which is
an element of $\cE^{(3)}_{g-1}(n)$. Otherwise,
let $(M'',S)=\Psi(M',c_3)$, and set
$\Psi(M,\tau):=(M'',S\cup\{v_1,v_2\})$.
Note that this recursive algorithm necessarily stops, since the genus of the map
decreases and since there are no trisections in unicellular maps of genus $0$ (plane
trees).
Thus this procedure yields recursively a mapping that associates to a 
unicellular map
$M$ with a marked trisection $\tau$ another unicellular map $M''$ of a smaller genus, with a
set $S''$ of marked vertices (namely the set of vertices which have been
involved in a slicing at some point of the procedure). The set $S''$ of marked
vertices necessarily has odd cardinality, as easily
seen by induction. Moreover, it is clear that the underlying graph of $M$
coincides with the underlying graph of $M''$ in which the vertices of $S''$
have been identified together into a single vertex.
%Hence if we admit that $\Psi$ is a bijection between the two sides of
%\eqref{eq:trisectionrec-maps}, it also satisfies the other stated properties.
%EF9 reecrit cette phrase legerement
One can show that $\Psi$ is a bijection~\cite{Ch09}, with an  
explicit inverse mapping. 

%VF : j'ai enlevé ce paragraphe et rajouté la phrase ci-dessus 
%pour gagner de la place.
%À mon avis, la description d'une direction suffit à donner une bonne idée au
%lecteur de la construction, et, avec ou sans le paragraphe suivant, s'il
%veut vraiment être convaincu que ça marche, il faut aller voir le papier
%de Guillaume.
%Before ending our sketch of the construction of \cite{Ch09}, let us say a few
%words about the reverse construction, that constructs a unicellular map $M$
%with a marked trisection, starting from a unicellular map $M'$ of smaller genus
%with a marked set of vertices $S$. We already said how to obtain the underlying
%graph of $M$: take the graph of $M'$ and identify all the vertices of $S$ into
%a single vertex. So to define the map $M$ it only remains to define the cyclic ordering around each
%vertex: for vertices not belonging to $S$, we keep the same
%ordering as in $M'$. For others, the cyclic ordering is defined recursively
%from the one of $M'$, thanks to a recursive gluing procedure that ``mimics'',
%at reverse, the successive slicings we have performed here: we refer to
%\cite{Ch09} for a complete description of this gluing procedure.

\subsection{Recursive decomposition of \cdec trees} \label{subsec:apresrappel}
We now propose a recursive method to decompose \cdec trees, which can be seen
as  parallel to the decomposition of unicellular maps given in the previous
section. Denote by $\cC(n)$ (resp. $\cC_g(n)$) 
the set of C-permutations on $n$ elements (resp. on $n$ elements and of genus $g$). 
A \emph{signed sequence} of integers is a pair $(\eps,S)$ where $S$ is an
integer  sequence and $\eps$ is a sign, either $+$ or $-$. We will often
write signed sequences with the sign preceding the sequence as a exponent, such as $^\epsilon S$. 

\begin{figure}
\begin{center}
\includegraphics[width=12cm]{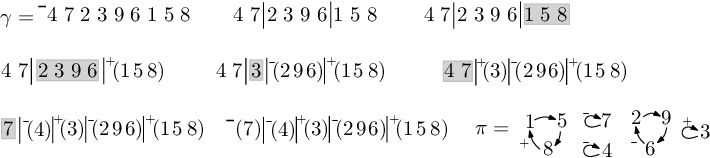}
\end{center}
\caption{The bijection between signed sequences and $C$-permutations.}
\label{fig:bij_min}
\end{figure}

\begin{lem}\label{lem:iso_signed}
Let $X$ be a finite non-empty set of positive integers. Then there is a bijection between signed
sequences of distinct integers from $X$ ---all elements of $X$ being present in the sequence--- and C-permutations on the set $X$. 
In addition the C-permutation has one cycle if and only if the signed
sequence has odd length  and starts with its minimal element. 
%Let $\cUp(n)$ be the set of permutations from $\cC(n)$ with a unique cycle,
%and where a non-minimal element is marked. Let $\cO(n)$ be the subset
%of C-permutations in $\cC(n)$ that have an odd number $2k+1>1$ of (odd) cycles 
%(note that $\cUp(n)$ and $\cC(n)$ are empty if $n$ is odd). 
 %Then
%$$
%\cUp(n)\simeq \cO(n).
%$$
\end{lem}
\begin{proof}
The bijection is illustrated in Figure~\ref{fig:bij_min}.
Starting from a signed sequence $^\epsilon S$, decompose $S$ into blocks according to the left-to-right minimum records.
Then treat the blocks successively from right to left. At each step, if the treated block $B$ has odd length, turn $B$ into
the signed cycle $^+(B)$; if $B$ has even length, move the second element $x$ of $B$ out of $B$, insert it at the end of the block preceeding $B$, 
and then turn $B$ into the signed cycle $^-(B)$. Update the block-decomposition (according to left-to-right minimum records) 
on the left of $B$ (it is very simple, two cases occur:
if $x$ is the minimum of the elements on the left of $B$, it occupies a single block; if not, $x$ is integrated at the end of the block on the left of $B$).    
At the end of the right-to-left traversal, the last treated block $B$ has odd length, we produce $^\epsilon(B)$ as the last signed cycle. The output is the $C$-permutation
$\pi$ made of all the signed cycles that have been produced during the traversal.

Conversely (read Figure~\ref{fig:bij_min} from right to left and bottom to top), starting from a $C$-permutation $\pi$, write $\pi$ as the ordered list of its signed cycles,
each cycle starting with its minimal element, and the cycles being ordered from left to right such that the minimal elements are in descending order.   
Record the sign $\epsilon_0$ of the leftmost signed cycle $^{\epsilon_0}(B)$, and turn $(B)$ into the block $B$.  
Then treat the signed cycles from left to right (starting with the second one). 
At each step, let $^{\epsilon}(B)$ be the treated signed cycle and let $B'$ be the block to the left of $^{\epsilon}(B)$. 
Turn $^{\epsilon}(B)$ into the block $B$, and in case $\epsilon=-$, move
the last element of $B'$  to the second position of $B$ (this possibly makes $B'$ empty, in which case we erase $B'$ of the current list of blocks). 
At the end, we get an ordered list of blocks, which can be seen as a sequence $S$.   
The output is the signed sequence $^{\epsilon_0}S$.

It is easy to see that the two mappings $\Phi$ (from signed sequences to $C$-permutations) and $\Psi$ (from $C$-permutations to signed sequences) are inverse
of each other; indeed these two mappings consist of a sequence of steps that operate on hybrid structures (a sequence of blocs followed by a sequence
of signed cycles, these all start with their minimal element, and the minimal elements decrease from left to right),  
each step of $\Phi$ (resp. $\Psi$) increases (resp. decreasing) by $1$ the number of signed cycles in the hybrid structure,
and the step of $\Phi$ with $i$ signed cycles is the inverse of the step of $\Psi$ with $i+1$ signed cycles. 
\end{proof}

An element of a C-permutation is called \emph{non-minimal} if it is
not the minimum in its cycle. Non-minimal elements play the same role for
C-permutations 
% GC2 added:
(and C-decorated trees)
as trisections for unicellular maps. Indeed, a C-permutation 
% GC2 changed %in $\cC_g(n)$ 
of genus $g$
has
$2g$ non-minimal elements (compare with Lemma~\ref{lem:trisectionscount}), and
moreover we have the following analogue of Proposition~\ref{prop:Eg_decomp}:
\begin{prop}\label{prop:Tg_decomp}
For $k\geq 1$, denote by $\cT_g^{(2k+1)}(n)$ the set of \cdec trees from $\cT_g(n)$
in which a set of $2k+1$ cycles is distinguished. Then for $g>0$ and $n\geq 0$,
$$
2g\ \!\cT_g(n)\simeq    \cT_{g-1}^{(3)}+\cT_{g-2}^{(5)}+\cT_{g-3}^{(7)}+\cdots+\cT_{0}^{(2g+1)}.
$$
In addition, if $\gamma$ and $(\gamma',S')$ are in correspondence,
then the underlying graph of $\gamma$ is obtained from the underlying graph of $\gamma'$
by merging the vertices corresponding to cycles from $S'$ into a single vertex. 
\end{prop}
\begin{proof}
For $k\geq 1$ let $\cC_g^{(2k+1)}(n)$ be the set of C-permutations from $\cC_g(n)$
where a subset of $2k+1$ cycles are marked. Let $\cCgc(n)$ be the set
of C-permutations from $\cC_g(n)$ where a non-minimal element is marked. Note that $\cCgc(n)\simeq 2g\ \!\cC_g(n)$ since a C-permutation in $\cC_g(n)$ has $2g$ non-minimal elements.
%EF9 \sum_{k=1}^g plutot que \sum_{k\geq 1}

%GC2013-> J'ai mis "we claim that" pour faire plaisir au referee et j'ai
%commencé ici le nouveau paragraphe
We now claim that $\cCgc(n)\simeq\sum_{k=1}^g\cC_{g-k}^{(2k+1)}(n)$.
%VF10: j'ai reformulé ce qui suit, il n'était pas écrit que l'on devait écrire 
%le cycle en commençant par l'élément marqué pour appl le lemme 3.
Indeed starting from $\gamma\in\cCgc(n)$,
write the signed cycle containing the marked element $i$ of $\gamma$ as a signed sequence
beginning with $i$ and apply Lemma~\ref{lem:iso_signed} to this signed sequence:
this produces a collection of $(2k+1)\geq 3$ signed cycles of odd length,
which we take as the marked cycles.

We have thus shown that $2g\ \!\cC_g(n)\simeq\sum_{k=1}^g\cC_{g-k}^{(2k+1)}(n)$. 
Since by definition $\cT_g(n)=\cE_0(n)\times\cC_g(n+1)$, 
we conclude that $2g\ \!\cT_g(n)\simeq\sum_{k=1}^g\cT_{g-k}^{(2k+1)}(n)$. The statement
on the underlying graph just follows from the fact that the procedure in Lemma~\ref{lem:iso_signed} merges
the marked cycles into a unique cycle. 
\end{proof}

% GC: ai changé le titre de la section, pour que le lecteur pressé s'arrête à cet endroit
% (Éric avait mis ``the bijection'')
% GC2 : ajouté stochastic pour faire mieux passer l'idée
\subsection{The main result} 
\begin{theo}\label{theo:main}
For each non-negative integers $n$ and $g$ we have
$$
2^{n+1}\cE_g(n)\simeq\cT_g(n). 
$$
In addition the cycles of a \cdec tree naturally correspond to the vertices 
of the associated unicellular map, in such a way that the respective underlying
graphs are the same. 
\end{theo}
\begin{proof}
The proof is a simple induction on $g$, whereas $n$ is fixed. 
The case $g=0$ is obvious, as there are $2^{n+1}$ different  $C$-permutations of size
$(n+1)$ and genus $0$, corresponding to the $2^{n+1}$ ways of giving signs to the identity
permutation. Let $g>0$. The induction hypothesis ensures
that for each $g'<g$, $2^{n+1}\cE_{g'}^{(2k+1)}(n)\simeq \cT_{g'}^{(2k+1)}(n)$, where 
the underlying graphs (taking marked vertices vertices into account) of corresponding
objects are the same. Hence, by Propositions~\ref{prop:Eg_decomp} and~\ref{prop:Tg_decomp}, we have
$2g\ \!2^{n+1}\cE_g(n)\simeq 2g\ \!\cT_g(n)$, where the underlying graphs of corresponding
objects are the same. Finally, one can extract
from this $2g$-to-$2g$ correspondence a $1$-to-$1$ correspondence 
(think of extracting a perfect matching from a $2g$-regular
bipartite graph, which is possible according to Hall's marriage theorem).
And obviously the extracted $1$-to-$1$ correspondence, which realizes   
 $2^{n+1}\cE_g(n) \simeq \cT_g(n)$, also preserves the underlying graphs.    
\end{proof}

%GC: j'ai rajouté des longueurs dans cette section
% quoi qu'il en soit j'ai l'impression qu'il faudrait sérieusement la racourcir
% (on doit pouvoir être moins précis mais faire passer les idées).
\subsection{A fractional, or stochastic, formulation}
Even if this does not hinder enumerative applications to be detailed in the next section,   
we do not know of an effective (polynomial-time) way to implement the 
bijection of Theorem~\ref{theo:main}; indeed the last step of the proof
is to extract a perfect matching from a $2g$-regular bipartite graph
whose size is exponential in $n$. 

What can be done effectively 
%EFMay ajout parenthese cf demande referee
(in time complexity $O(gn)$) 
is a \emph{fractional} 
formulation of the bijection. 
For a finite set $X$, let $\mathbb{C}\langle X\rangle$ be the set
of linear combinations of the form $\sum_{x\in X}u_x\cdot x$, where the $x\in X$
are seen as independent formal vectors, and the coefficients $u_x$ are in $\mathbb{C}$.
Let $\mathbb{R}^+_{1}\langle X\rangle\subset\mathbb{C}\langle X\rangle$ 
be the subset of linear combinations where the coefficients
are nonnegative and add up to $1$. Denote by $\mathbf{1}_X$ the
vector $\sum_{x\in X}x$. For two finite sets $X$ and $Y$, 
a \emph{fractional mapping} 
from $X$ to $Y$ is a linear mapping $\phi$ from $\mathbb{C}\langle X\rangle$
to $\mathbb{C}\langle Y\rangle$ such that the image of each $x\in X$ is in
$\mathbb{R}^+_1\langle Y\rangle$;
%EF9 reecrit legerement
the set of elements of $Y$ whose coefficients in $\phi(x)$ are strictly positive is called
the \emph{image-support} of $x$. 
%EF ajout de cette phrase
Note that $\phi(x)$ identifies to a probability
distribution on $Y$; a ``call to $\phi(x)$'' is meant as picking up $y\in Y$
under this distribution. 
A fractional mapping is \emph{bijective} if $\mathbf{1}_X$ is mapped to
$\mathbf{1}_Y$, and is \emph{deterministic} if each $x\in X$ is mapped to
some $y\in Y$. 
% GC: enlevé ça pour gagner de la place.
%Note that a deterministic fractional mapping (resp. fractional
%bijection) is just the linear lift of a classical mapping (resp. classical
%bijection) from $X$ to $Y$. 
Note 
%also 
that, if there is a fractional bijection
from $X$  to $Y$, then $|X|=|Y|$ (indeed in that case the matrix of $\phi$
is bistochastic). 
%Given a fractional bijection $\phi$ and an element $ x \in X$, the formula
%$\phi(x)=\sum u_y y$ defines a probability measure $\mu_x$ on $Y$ (the measure that
%gives mass $u_y$ to $y$). Moreover, if one sample $x$ uniformly at random over
%$X$, and then chose $y$ according to $\mu_x$, then $y$ is uniformly distributed
%on $Y$. We say that a fractional mapping $\phi$ is \emph{linear-time effective} if one can
%compute the image of an element $x\in X$ in linear time (according to a chosen
%size parameter, which in our case will be $n$). In this case, what precedes gives a linear time
%random sampler that outputs  the uniform measure on the set $Y$ , assuming that one
%already is equipped with such a sampler for the set $X$.

%EF donne la complexite; c'est peut etre plus rapide, O(n) ?
One can now formulate by induction 
on the genus an effective (the cost of a call is $O(gn)$) fractional bijection
 from $2^{n+1}\cE_g[n]$ to $\cT_g(n)$,
 and similarly from $\cT_g[n]$ to $2^{n+1}\cE_g(n)$. 
%(by \emph{linear-time}, we mean
 %that the image $\phi(x)$ of any element $x\in X$ can be computed in linear-time
% w.r.t. $n$).
The crucial property is that, for $k\geq 1$ and $E$, $F$ finite sets,
 if there is a fractional bijection $\Phi$ from $kE$
to $kF$ then one can \emph{effectively} 
derive from it a fractional bijection $\widetilde{\Phi}$ from $E$ to $F$:
for $x \in E$, just define $\widetilde{\Phi}(x)$
as $\tfrac1{k}(\iota(\Phi(x_1))+\cdots+\iota(\Phi(x_k)))$,
where $x_1,\ldots,x_k$ are the representatives of $x$ in $kE$,
and where $\iota$ is the projection from $kF$ to $F$.  
%EF6: ajout de cette phrase, changement leger de la phrase precedente
In other words a call to $\widetilde{\Phi}(x)$ consists in picking up a representative
$x_i$ of $x$ in $kE$ uniformly at random and then calling $\Phi(x)$.  
Hence by induction on $g$, Propositions~\ref{prop:Eg_decomp} and~\ref{prop:Tg_decomp} (where the stated
combinatorial isomorphisms are effective) 
ensure that there is an effective fractional bijection
from  $2^{n+1}\cE_g(n)$ to $\cT_g[n]$ and similarly from $\cT_g[n]$
to $2^{n+1}\cE_g[n]$, such that if $\gamma'$ is in the image-support of $\gamma$
then the underlying graphs of $\gamma$ and $\gamma'$ are the same.

%GC: voilà pour l'aspect stochastique:
Note that, given an effective fractional bijection between two sets $X$ and
$Y$, and a uniform random sampling algorithm on the set $X$, one obtains
immediately a uniform random sampling algorithm for the set $Y$.
% Je ne détaille pas plus, avec un peu d'intuition probabiliste ça doit être clair
In the next section, we will use our bijection to prove several
enumerative formulas for unicellular maps, starting from elementary results on
the
enumeration of trees or permutations.
In all cases, we will also be granted with a
uniform random sampling algorithm for the corresponding unicellular maps, though we
will not emphasize this point in the rest of the paper.
%GC: c'est vrai ce que je viens d'écrire ?

\section{Counting formulas for unicellular maps}
\label{SectCounting}

%GC2: changé cette phrase, pour gagner 14 caractères et ainsi faire passer la
%proposition Harer-Zagier en bas de page!
%The big advantage of \cdec trees as a model to represent unicellular maps
%is that they are simple combinatorial objects. 
It is quite clear that \cdec trees are much simpler combinatorial objects than
unicellular maps. 
%\cdec trees give us a very simple combinatorial representation of
%unicellular maps.
In this section, we use them to
give bijective proofs of several known formulas concerning unicellular maps. We
focus on the Lehman-Walsh and the Goupil-Schaeffer formulas, and the Harer-Zagier
recurrence, of which bijective proofs were long-awaited.
%GC: j'ai osé ``longly awaited''. N'ayons honte de rien!
%VF: pas mal, mais word reference ne connait pas longly et propose long-awaited.
%Sinon, il y a aussi ``eagerly awaited'' mais ça fait peut-être beaucoup...
%GC2013: ajout de la modif de Valentin
%%We also sketch a bijective proof of the Harer-Zagier summation formula (different bijective proofs are already known for such summation formulas). 
We also give new bijective proofs of several summation formulas
(for which different bijective proofs are already known):
the Harer-Zagier summation formula and a refinement of it,
a formula due to Jackson for bipartite maps
and its refinement due to A.~Morales and E.~Vassilieva.
We finally consider an identity involving {\em covered maps},
obtained originally from a difficult bijection by the first author and O.~Bernardi
and which can be explained easily thanks to our new bijection.
%
%EF9 reecrit legerement
%(prototype for a family of formulas for which bijective proofs were already known). 
We
insist on the fact that all these proofs are elementary consequences of our
main bijection (Theorem~\ref{theo:main}).

\subsection{Two immediate corollaries}\label{SubsectImmediateCorol}
The set $\cT_g(n)=\cE_0(n)\times \cC_g(n+1)$ is the product of two sets
that are easy to count.
Precisely, let $\eps_g(n)=|\cE_g(n)|$ and $c_g(n)=|\cC_g(n)|$. 
Recall that $\eps_0(n)=\Cat(n)$, where $\Cat(n):=\frac{(2n)!}{n!(n+1)!}$ is 
the $n$-th Catalan number. 
Therefore Theorem~\ref{theo:main} gives
$
\eps_g(n)=2^{-n-1}\Cat(n)c_g(n+1).
$

%EF9 reecrit
%This immediately yields
%VF10: ai changé ça, ce n'est pas une conséquence de ce qu'il y a au dessus
One gets easily a closed form  for $c_g(n+1)$ (by summing over all
possible cycle types) and an explicit formula for the generating series,  
 thereby recovering two classical results for the enumeration of unicellular maps.

%VF10: rajouté un peu d'explication.
Every partition of $n+1$ in $n+1-2g$ odd parts writes as
$1^{n+1-2g-\ell}3^{m_1}\ldots (2k+1)^{m_k}$
for some partition
$\gamma=(\gamma_1,\ldots,\gamma_{\ell})=1^{m_1}\ldots k^{m_k}$
of $g$.
The number $a_{\gamma}(n+1)$ of 
permutations of \hbox{$n+1$} elements with 
cycle-type equal to $1^{n+1-2g-\ell}3^{m_1}\ldots (2k+1)^{m_k}$
is classically given by 
%EF6 ai mis cette equation en display
$$
a_{\gamma}(n+1)=\frac{(n+1)!}{(n+1-2g-\ell)!\prod_im_i!(2i+1)^{m_i}}, 
$$
and the number of C-permutations with this cycle-type is just 
$a_{\gamma}(n+1)2^{n+1-2g}$ (since each cycle has 
% GC2 changed %weight $2$
$2$ possible signs).
Hence, we get the equality
%VF10: mis des $$
$$
c_g(n+1) = 2^{n+1-2g} \sum_{\gamma \vdash g}a_{\gamma}(n+1).
$$
We thus recover:
\begin{prop}[Walsh and Lehman~\cite{LeWa72}]
The number $\eps_g(n)$ is given by
$$
\eps_g(n)=\frac{(2n)!}{n!(n+1-2g)!2^{2g}}\sum_{\gamma \vdash g}\frac{(n+1-2g)_{\ell}}{\prod_im_i!(2i+1)^{m_i}},
$$
% GC2: au mis une formule close pour le produit, pour faire passer la prop
% Harer-Zagier en bas de page
%where $(x)_k=x(x-1)\cdots(x-k+1)$, 
where $(x)_k=\prod_{j=0}^{k-1}(x-j)$, 
$\ell$ is the number of parts of $\gamma$,
and $m_i$ is the number of parts of length $i$ in $\gamma$. 
\end{prop}

Define the exponential generating function 
%VF10: put that on one line, it was ugly
$$C(x,y):=\sum_{n,g}\tfrac1{(n+1)!}c_g(n+1)y^{n+1}x^{n+1-2g}$$
of % signed cycles of odd length. %EF6 plutot: of C-permutations ?
C-permutations where $y$ marks the number of elements,
which are labelled, and $x$
marks the number of cycles.
%VF10: coupé la phrase et ajouté une référence au Flajolet,
%tout le monde n'est pas tombé dedans quand il était petit!!
%EFMay detaille un peu plus cf remarque referee
Since a  C-permutation is a set of signed cycles of odd lengths, $C(x,y)$ is given by 
%EFMay ai remplace k\geq 1 par k\geq 0
$$
C(x,y)=\exp\Big(2x\sum_{k\geq 0}\frac{y^{2k+1}}{2k+1}\Big)-1.
$$
Indeed the sum in the parenthesis is the generating function of cycles of odd lengths, the factor $2$ is for the signs of cycles, the $\exp$ means that we take a set
of such signed cycles, and the $-1$ means that the set is non-empty 
(see {\it e.g.} \cite[Part A]{FlajoletSedgewickAnalyticCombinatorics} for a
general presentation of the methodology to translate classical combinatorial set operations into generating function expressions, in particular page 120 for the application
to permutations seen as sets of cycles). Since $\sum_{k\geq 0}\frac{y^{2k+1}}{2k+1}=\frac1{2}\log\Big(\frac{1+y}{1-y}\Big)$, the expression simplifies to
% GC2 removed % given by
$$
C(x,y)=\exp\Big(x\log\Big(\frac{1+y}{1-y}\Big)\Big)-1=\Big(\frac{1+y}{1-y}\Big)^x-1.
$$
Since $c_0(1)=2$ and $\frac{1}{(n+1)!}c_g(n+1)=\frac{2^{n+1}n!}{(2n)!}\eps_g(n)=\frac{2}{(2n-1)!!}\eps_g(n)$ 
for $n\geq 1$, we recover: 

\begin{prop}[Harer-Zagier series formula~\cite{HaZa86,ZvLa97}]
The generating function 

$\displaystyle E(x,y):=1+2xy+2\sum_{g\geq 0,n>0}\frac{\eps_g(n)}{(2n-1)!!}y^{n+1}x^{n+1-2g}$ is given by
$$
 E(x,y)=\Big(\frac{1+y}{1-y}\Big)^x.
$$
\end{prop}

\subsection{Harer-Zagier recurrence formula}
Elementary algebraic manipulations on the expression of $E(x,y)$ yield
a very simple recurrence satisfied by $\eps_g(n)$, 
known as the Harer-Zagier recurrence formula 
(stated in Proposition~\ref{prop:hazaA1} hereafter).  
We now show that 
the model of \cdec trees makes it possible to derive this recurrence
directly from a combinatorial isomorphism, 
%EFMay changé pour rendre compte du fait que Remy opere sur les arbres binaires
%that generalizes R\'emy's beautiful
%bijection for plane trees~\cite{Remy}.
that generalizes R\'emy's beautiful 
bijection~\cite{Remy} formulated on plane trees.

It is convenient here
to consider \cdec trees as \emph{unlabelled structures}: precisely
we see a \cdec tree as a plane tree where the vertices are
partitioned into parts of odd size, where each part carries a sign $+$ or $-$, 
and such that the vertices in each part
are cyclically ordered (the C-permutation can be recovered by numbering
the vertices of the tree according to a left-to-right depth-first traversal), 
think of Figure~\ref{fig:C-tree}(c) where the labels have been taken out. 
%EF6: explique que on doit prendre 2n+1 coins
We take here the convention that a plane tree with 
$n$ edges has $2n+1$ corners, considering that the sector of the root has two
corners, one on each side of the root.   

We denote by $\cP(n)=\cE_0(n)$ the set of plane trees with $n$ edges,
and by $\cPv(n)$ (resp. $\cPc(n)$) the set of plane trees with $n$ edges
where a vertex (resp. a corner) is marked. R\'emy's procedure,
shown in Figure~\ref{fig:Remy}, realizes the isomorphism $\cPv(n)\simeq 2\ \!\cPc(n-1)$, 
or equivalently 
\begin{equation}\label{eq:remyP}
(n+1)\cP(n)\simeq 2(2n-1)\cP(n-1).
\end{equation}

%GC: belle figure, mais à mon avis il y a de la marge pour la réduire 
% (soit gagner 2cm en hauteur, soit juste gagner en largeur et pousser la
% légende sur le côté)... enfin nous verrons!
\begin{figure}
\begin{center}
%\begin{minipage}{10cm}
\includegraphics[width=11cm]{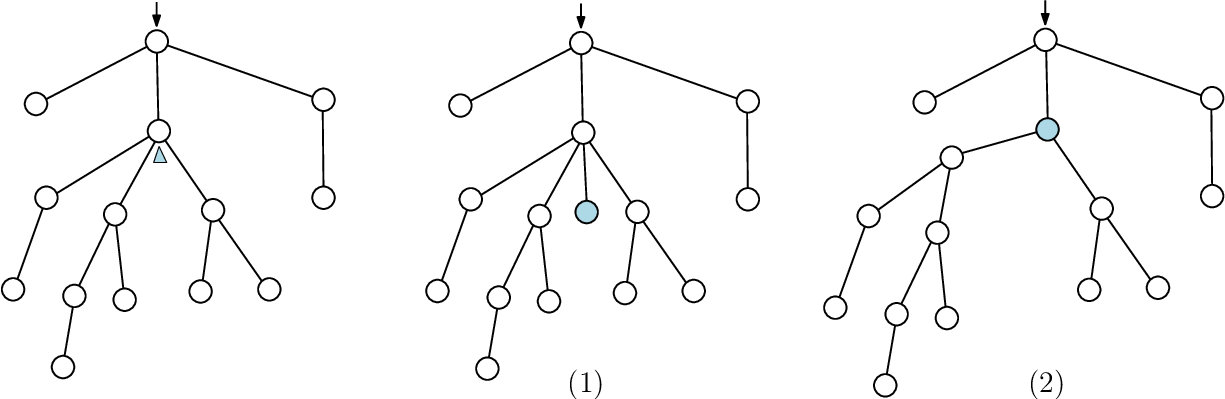}
%\end{minipage}
%\hfill
%\begin{minipage}{4cm}
\caption{R\'emy's procedure gives two ways to obtain a plane tree with $n$ edges and a marked vertex $v$ 
%EF6: ai ajoute des precisions entre parentheses
from a plane tree with $n-1$ edges and a marked corner: (1) in the first way 
(replacing  the marked corner by a leg) 
$v$ is a leaf, (2) in the second way (stretching an edge to carry the subtree on the left of the marked corner) $v$ is a non-leaf.}
\label{fig:Remy}
%\end{minipage}
\end{center}
\end{figure}
%VF5: remise en page
Let $\cTgv(n)$ be the set of \cdec trees from $\cT_g(n)$ where a vertex is marked. 
Let $\cA$ (resp. $\cB$) be the subset of objects in $\cTgv(n)$ where the signed cycle containing
the marked vertex has length $1$ (resp. length greater than $1$).
Let $\gamma\in\cTgv(n)$, with $n\geq 1$.  
If $\gamma\in\cA$, record the sign of the $1$-cycle containing $v$ and then
apply R\'emy's procedure to the plane tree with respect to $v$ (so as to delete $v$).
This reduction, which does not change the genus, yields 
%GC2: j'ai dû mettre des éq dans le texte pour gagner 1cm sur cette page... 
% si on gagne encore 2cm dans ce paragraphe, alors  par rebonds successifs
% on gagne 10 cm à la fin de l'article et on tient dans les 12 pages !
$\displaystyle\cA\simeq 2\cdot 2(2n-1)\cT_g(n-1)$.
If $\gamma\in\cB$, let $c$ be the cycle containing the marked vertex $v$;
$c$ is of the form $(v,v_1,v_2,\ldots,v_{2k})$ for some $k\geq 1$. Move
$v_1$ and $v_2$ out of $c$ (the successor of $v$ becomes the former successor
of $v_2$). Then apply R\'emy's procedure twice, firstly with respect to $v_1$
(on a plane tree with $n$ edges), secondly with respect to $v_2$ (on a plane tree
with $n-1$ edges). This reduction, which decreases the genus by $1$, yields
$\cB\simeq 2(2n-1)2(2n-3)\cT_{g-1}^{\mathrm{v}}(n-2)$, hence
$
\cB\simeq 4(n-1)(2n-1)(2n-3)\cT_{g-1}(n-2)$. 
Since $\cTgv(n)=\cA+\cB$ and $\cTgv(n)\simeq(n+1)\cT_g(n)$, we finally obtain the isomorphism
\begin{equation}\label{eq:recT}
(n+1)\cT_g(n)\simeq 4(2n-1)\cT_g(n-1)+4(n-1)(2n-1)(2n-3)\cT_{g-1}(n-2),
\end{equation}
which holds for any $n\geq 1$ and 
$g\geq 0$ (with the convention $\cT_g(n)=\emptyset$ if $g$
or $n$ is negative).  Since $2^{n+1}\cE_g(n)\simeq\cT_g(n)$, we recover:
\begin{prop}[Harer-Zagier recurrence formula~\cite{HaZa86,ZvLa97}]
The coefficients $\eps_g(n)$ satisfy the following recurrence relation valid 
for any $g\geq 0$ and $n\geq 1$ 
(with $\eps_0(0)=1$ and $\eps_g(n)=0$ if $g<0$ or $n<0$):
$$
(n+1)\eps_g(n)=2(2n-1)\eps_g(n-1)+(n-1)(2n-1)(2n-3)\eps_{g-1}(n-2).
$$
\end{prop}
To the best of our knowledge this is the first proof of the Harer-Zagier recurrence
formula that directly follows from a combinatorial isomorphism. 
The isomorphism~\eqref{eq:recT} also provides
a natural extension to arbitrary genus of R\'emy's
isomorphism~\eqref{eq:remyP}.\\

\subsection{Refined enumeration of bipartite unicellular maps}
In this paragraph, we explain how to recover a formula due to
A. Goupil and G. Schaeffer~\cite[Theorem 2.1]{GoSc00}
%GC2 added
from our bijection.
Let us first give a few definitions.
A graph is \emph{bipartite} if its vertices can be
% GC2 changed %partitioned into black and white vertices 
colored in black and white
such that
each edge connects a black and a white vertex. 
If the graph has a root-vertex $v$, then $v$ is required to be black;
thus, if the graph is also connected, then such a bicoloration of the vertices
is unique.
From now on, a connected 
bipartite graph with a root-vertex is assumed to be endowed
with this canonical bicoloration. 

The degree distribution of a map/graph is
the sequence of the degrees of its vertices taken in decreasing order 
(it is a partition of $2n$, where $n$ is the number of edges).
If we consider a bipartite map/graph, we can consider separately
the {\em white vertex degree distribution} and 
the {\em black vertex degree distribution}, which
 are two partitions of $n$.%EF enleve with obvious definitions.

%EF met definition du genre plus tot
Let $\ell,m,n$ be positive integers such that $n+1-\ell-m$ is even.
 Fix two partitions $\lambda$, $\mu$ of $n$ of respective lengths $\ell$ and $m$. 
We call $\Bi(\lambda,\mu)$ the number of bipartite unicellular maps, 
with white (resp. black) vertex degree distribution $\lambda$ (resp. $\mu$).
The corresponding genus is $g=(n+1-\ell-m)/2$. 

The purpose of this paragraph is to compute $\Bi(\lambda,\mu)$.
It will be convenient to change a little bit the formulation of the problem
and to consider \emph{labelled maps} instead of the usual non-labelled maps:
 a \emph{labelled map}
% GC2: je n'aime pas beaucoup ``Cayley'', qui pour moi est synonyme de
% ``étiqueté mais non plongé''. Je préfère tout simplement ``labelled map''
% (et les trois caractères en plus sont compensés par l'absence de footnote)
%\footnote{This denomination has been chosen as an analogy with Cayley trees.
%However, we recall that all maps and trees considered here are rooted, unlike the usual Cayley trees.}
is a map whose vertices are labelled with integers $1,2,\cdots$.
If the map is bipartite, we require instead that
the white and black vertices are labelled
separately (with respective labels $w_1,w_2,\cdots$ and $b_1,b_2,\cdots$).
The degree distribution(s) of a labelled map (resp. bipartite labelled
 map) with $n$ edges can be seen
as a composition of $2n$ (resp. two compositions of $n$).
For $\II=(i_1,\cdots,i_{\ell})$ and $\JJ=(j_1,\cdots,j_m)$ two compositions of $n$, 
we denote by $\BiL(\II,\JJ)$ the number of labelled bipartite unicellular maps
with white (resp. black) vertex degree distribution $\II$ (resp. $\JJ$).
The link between $\Bi(\lambda,\mu)$ and $\BiL(\II,\JJ)$
is straightforward:
$ \BiL(\II,\JJ) = m_1(\lambda)! m_2(\lambda)! \cdots m_1(\mu)! m_2(\mu)! \cdots
\Bi(\lambda,\mu),$
where $\lambda$ and $\mu$ are the sorted versions of $\II$ and $\JJ$. We now
recover the following formula:

% EF enleve In the rest of this section, we give the first combinatorial proof of the following formula:
% GC2: chez moi ça ne compile pas si je mets le cite en argument de la proposition
% alors j'ai créé une commande.
\newcommand{\citationProblematique}{\cite[Theorem 2.1]{GoSc00}}
\begin{prop}[Goupil and Schaeffer \citationProblematique]\label{prop:GS}
%EF enleve les facteurs 2^{n+1} de chaque cote, corrige 2^{-g} en 2^{-2g}
\begin{multline}
     \BiL(\II,\JJ) 
    = 2^{-2g} \cdot n 
%GC8 : j'ai mis tous les facteurs qui dépendaient de g_1, g_2 
% à l'intérieur de la somme...
% :-)
\cdot \sum_{g_1+g_2=g} 
 (\ell + 2g_1 -1)! (m + 2g_2 -1)! \\
 \sum_{p_1+\cdots+p_\ell=g_1 \atop q_1+\cdots+q_m=g_2}
\prod_{r=1}^\ell \frac{1}{2p_r+1} \binom{i_r -1}{2 p_r} 
\prod_{r=1}^m \frac{1}{2q_r+1} \binom{j_r -1}{2 q_r}. \label{eq:GS}
\end{multline}

\end{prop}
\begin{proof}
%The fact (used in paragraph \ref{SubsectImmediateCorol})
 %that the number of (bipartite) plane trees with $n$ edges is $\Cat_{n+1}$
%(all trees are bipartite) has been refined by Goulden and Jackson
%(\cite[Theorem 2.2]{GJ92},
% GC2: pour gagner une référence j'ai tenté la parenthèse suivante
% es-tu d'accord, Éric
%(see also \cite{Ba11} for a bijective proof of an equivalent statement):
%which can also be proved combinatorially using the cycle lemma):
%let $\lambda$ and $\mu$ be two partitions of the same integer $n$ with 
%$\ell(\lambda)+\ell(\mu)=n+1$,
%the number of bipartite plane trees 
%with white (resp. black) vertex degree distribution $\lambda$ (resp. $\mu$) is
%\[Bi(\lambda,\mu) = n \cdot \frac{(\ell(\lambda)-1)!(\ell(\mu)-1)!}{
%m_1(\lambda)! m_2(\lambda)! \cdots m_1(\mu)! m_2(\mu)! \cdots}. \]
%The statement in terms of labelled (rooted) plane tree is even nicer
%if $\II=(i_1,\cdots,i_\ell)$ and $\JJ=(j_1,\cdots,j_m)$ are two compositions of the same integer $n$,
%with $\ell+m=n+1$, then
%EF raccourci le debut de la preuve
For $g=0$ the formula is simply
\begin{equation}\label{eq:g0}
   \BiL(\II,\JJ) = n (\ell-1)! (m-1)!,
\end{equation}
which can easily be established by a bivariate version of the cycle lemma,
see also~\cite[Theorem 2.2]{GJ92}.  
(Note that, in that case, the cardinality only depends on the lengths of $\II$ and $\JJ$.) 

%From now on, we fix two compositions $\II$ and $\JJ$ of size $n$ and respective length
%$\ell$ and $m$.
%Denote $g=(n+1-\ell+m)/2$, we suppose that it is an integer.
We now prove the formula for arbitrary $g$. 
Consider some lists $\pp=(p_1,\cdots,p_\ell)$ and $\qq=(q_1,\cdots,q_m)$ of nonnegative
integers 
with total sum $g$: let $g_1=\sum p_i$ and $g_2= \sum q_i$.
We say that a composition $\HH$ \emph{refines} $\II$ along $\pp$ if $\HH$ is of the form 
$(h_1^1,\cdots,h_1^{2p_1+1},\cdots,h_\ell^1,\cdots,h_\ell^{2p_\ell+1}),$
with $\sum_{t=1}^{2p_r+1} h_r^t = i_r$ for all $r$ between $1$ and $\ell$.
Clearly, there are $\prod_{r=1}^\ell \binom{i_r -1}{2 p_r}$ such compositions $\HH$.
One defines similarly a composition $\KK$ refining $\JJ$ along $\qq$.  

Consider now the set of labelled bipartite plane trees of vertex degree distributions $\HH$ and
$\KK$, where $\HH$ (resp. $\KK$) refines $\II$ (resp. $\JJ$) along $\pp$ (resp. $\qq)$.
By~\eqref{eq:g0}, there are
$n \cdot (\ell + 2g_1 -1)! (m + 2g_2 -1)! $
% GC2 removed %such 
trees for each pair $(\HH,\KK)$, so
%GC2: inversé texte pour gagner une ligne
in total, with $\II$, $\JJ$, $\pp$ and $\qq$ fixed, the number of such trees
is:
\begin{equation}\label{EqNbTrees}
    n \cdot (\ell + 2g_1 -1)! (m + 2g_2 -1)! 
\prod_{r=1}^\ell \binom{i_r -1}{2 p_r} \prod_{r=1}^m \binom{j_r -1}{2 q_r}.
\end{equation}

As the parts of $\HH$ (resp. $\KK$) are naturally indexed by pairs of integers,
we can see these trees as labelled by the set 
% GC2: très moche, j'ai mis dans le texte, mais à ce stade il faut vraiment
% gratter partout
%\[ \{w_r^t; 1 \leq r \leq \ell, 1 \leq t \leq 2 p_r+1\}
%\sqcup \{b_r^t; 1 \leq r \leq m, 1 \leq t \leq 2 q_r+1\}. \]
$\displaystyle \{w_r^t; 1 \leq r \leq \ell, 1 \leq t \leq 2 p_r+1\}
\sqcup \{b_r^t; 1 \leq r \leq m, 1 \leq t \leq 2 q_r+1\}. 
$
There is a canonical permutation of the vertices of the trees
with cycles of odd sizes and which preserves the bicoloration:
just send $w_r^t$ to $w_r^{t+1}$ (resp. $b_r^t$ to $b_r^{t+1}$),
where $t+1$ is meant modulo $2 p_r+1$ (resp. $2 q_r +1$).
If we additionally put a sign on each cycle, we get a \cdec tree (with labelled cycles)
that corresponds to a labelled bipartite map with white (resp. black) 
vertex degree distribution $\II$ (resp. $\JJ$). 
Conversely, to recover a labelled bipartite plane tree from such a \cdec tree, one has
to 
%The underlying graph associated to such a tree with the cycle structure above
%is a labelled bipartite graph with vertex degree distribution $\II$ and $\JJ$.
%Therefore, if we choose any signs for the cycles and apply the bijection
%of section \ref{SectBij}, we get a labelled bipartite map
%with vertex degree distribution $\II$ and $\JJ$.
%Conversely, a labelled bipartite map is sent by this bijection
%over a \cdec bipartite tree, whose cycles are labelled by $w_1,\cdots,w_\ell$
%and $b_1,\cdots,b_m$.
%To get a labelled bipartite plane tree from that structure, one has to
% GC2: idem, removed itemize.
%\begin{itemize}
%    \item forget the signs (hence a factor $2^{n+1 - g}$, when we count maps in terms of numbers of trees);
%    \item choose in each cycle which vertex gets the label $w_r^1$ (or $b_r^1$).
%        When we count maps in terms of trees, this yields a factor
%        $\left( \prod_{i=1}^\ell (2p_i+1) \prod_{i=1}^m (2q_i+1) \right)^{-1}.$
% \end{itemize}
% GC2- finalement ai reformulé.
%1. forget the signs (hence a factor $2^{n+1 - g}$, when we count maps in terms
%of numbers of trees); 2. choose in each cycle which vertex gets the label
%$w_r^1$ or $b_r^1$ (when we count maps in terms of trees, this yields a
%factor $\left( \prod_{i=1}^\ell (2p_i+1) \prod_{i=1}^m (2q_i+1) \right)^{-1}$
%).
choose in each cycle which vertex gets the label $w_r^1$ or $b_r^1$, and one
has to forget the signs of the $(n+1-g)$ cycles. This represents a factor
$2^{n+1-2g}\left( \prod_{r=1}^\ell (2p_r+1) \prod_{r=1}^m (2q_r+1)
\right)^{-1}$.%, when we count maps in terms of trees.

% GC2 : j'ai beau relire, je ne comprends pas le paragraphe suivant:
%       je ne sais plus trop où on en est et de quoi on parle
%       (mais je crois que je comprends la preuve, cela dit)
%Then we obtain a Cayley bipartite plane tree with vertex degree distribution $\HH$ and
%$\KK$, where $\HH$ (resp. $\KK$) can be obtained by splitting $\II$ (resp. $\JJ$)
%along {\em some} sequence $\pp$ (resp. $\qq$).
%Comparing this with equation \eqref{EqNbTrees}, we get the following formula 
%which corresponds to

% GC2: j'ai donc réécrit comme ça:
Multiplying \eqref{EqNbTrees} by the above factor, and summing over all
possible sequences $\pp$ and $\qq$ of total sum $g$, we conclude 
that the number of C-decorated trees
associated with labelled bipartite unicellular maps of white (resp. black)
vertex degree distribution $\II$ (resp. $\JJ$), is equal to $2^{n+1}$ times the right-hand side of
\eqref{eq:GS}. By Theorem~\ref{theo:main}, this number 
is also equal to $2^{n+1}\BiL(\II,\JJ)$. This ends the proof of Proposition~\ref{prop:GS}.
\end{proof}

This is the first combinatorial proof of~\eqref{eq:GS} 
(the proof by Goupil and Schaeffer involves representation theory of the symmetric group).
% GC2: removed
%In particular, it gives an efficient algorithm to do uniform random sampling of maps with prescribed
%degree distribution.
% ---> en raison de la discussion à la fin de la partie ``fractionnaire''
% EF: allege un peu cette phrase
% GC3: remis la référence aux interrogations de GS
Moreover, the authors of~\cite{GoSc00} found surprising that ``the two partitions
contribute independently
to the genus''. With our approach, this  is very natural, 
% GC2: why not ``quite obvious'' as before ?
% GC2 added
since the cycles are carried independently
by white and black vertices.
% pas super satisfait de cette dernière phrase mais voulais détailler un peu
% plus

\begin{remark}
    If we set $\JJ=(2^m)$, we find the number of monochromatic maps with
    $m$ edges and vertex-degree distribution $\II$.
    This explains why
    we did not consider separately the monochromatic and bipartite cases
    (as we do in the next section).
\end{remark}

%GC2: ajouté un s à formula dans le titre, même si on n'en traite qu'une
%explicitement

%VF5: j'aime pas trop le titre ``Summation formulas for colored maps''.
%Ces formules peuvent être vues soit comme formules pour des cartes colorées,
%soit comme sommatoires pour les cartes normales.
%Mais il me paraît inappropriés de mélanger les deux.
%Du coup, je change le titre
\subsection{Counting colored maps} 
\label{SubsectColoredMaps}

%VF5: comme il y a deux formules de plus que dans la version courte,
%j'ai essayé de structurer un peu le paragraphe..
%EF6 j'ai legerement reformule
In this paragraph, we deal with what was presented in the introduction
as the {\em first type} of formulas.
%EF9 legerement reecrit
These formulas give an expression for a certain \emph{sum} of coefficients counting 
 unicellular maps, the expressions being usually simpler than those
for the counting coefficients taken separately (like the Goupil-Schaeffer's
formula).    
These sums 
can typically be seen as 
counting formulas for {\em colored unicellular maps}
(where the control is on the number of colors, which gives
indirect access to the genus).   
%which are unicellular maps with additional information.\medskip

%VF: j'adapte ce paragraphe au fait que j'ai commenté la formule de Jackson
%We recover now two summation formulas, 
%one due to Harer-Zagier~\cite{HaZa86,ZvLa97} 
%for unicellular maps  (also very easily derived from the expression of $E(x,y)$) and one due to Jackson~\cite{Ja88} 
%for bipartite unicellular maps. These formulas have already been given several combinatorial proofs (in~\cite{La01,GoNi,BeHa} for the Harer-Zagier formula, in~\cite{ScVa08,BeHa} for Jackson's
%formula) using different principles.

%GC7: ici il y avait une sorte de faux titre qui trainait tout seul:
\subsubsection{A summation formula for unicellular maps.}
% je l'ai simplement viré.

We begin with Harer-Zagier's summation formula~\cite{HaZa86,ZvLa97} 
(which can also be very easily derived from the expression of $E(x,y)$).
In contrast to
%the Harer-Zagier recurrence, 
% GC3 suivi remarque de Valentin 
the formulas presented so far,
this one has already been given combinatorial 
proofs~\cite{La01,GoNi,BeHa} using different bijective constructions, 
%EF: je simplifie ce paragraphe pour ne pas insister trop en rappelant juste que pour la
%recurrence nous sommes les premiers a le faire. 
%GC3 moi j'aime bien insister !
but we want to insist on the fact that our construction gives
bijective proofs for all the formulas in a unified way.
%GC3 (j'ai raccourci sinon 14 pages)
%in a unified way
%bijective proofs for all the previously known formulas
% that enumerate unicellular maps.
\begin{prop}[Harer-Zagier summation formula~\cite{HaZa86,ZvLa97}]
\label{prop:hazaA1}
Let $A(v;n)$ be the number of unicellular maps with $n$ edges and $v$ vertices.
Then for $n\geq 1$
%Let $\eps(n,r)$ be the number of unicellular maps with $n$ edges
%where vertices are colored in $r$ possible colors, and such that each color $i\in\llbracket1,r\rrbracket$
%is used at least once. Then for $n,r$ positive integers
%$$
%\eps(n,r)=(2n-1)!!\ \!2^{r-1}\binom{r}{n-1}.
%$$
%Hence,
$$
\sum_{v}A(v;n)x^v=(2n-1)!!\sum_{r\geq 1}2^{r-1}\binom{n}{r-1}\binom{x}{r}.
$$
\end{prop}
\begin{proof}
To be comprehensive, let us begin by explaining
the well-known combinatorial reformulation
of this formula.
Let  $A_r(n)$ be the number of unicellular maps with $n$
edges, each vertex having a color in $\llbracket 1,r\rrbracket$, and each color in $\llbracket 1,r\rrbracket$
being used at least once.
Then 
\begin{equation}
    \label{HarerZagierCombi}
    \sum_{v} A(v;n) x^v=\sum_{r\geq 1} A_r(n) \binom{x}{r}.
\end{equation}
Indeed, both sides count the number of pairs $(M,\varphi)$,
where $M$ is a unicellular map with $n$ edges and $\varphi$ is 
a mapping from the vertex set of $M$ to a given set $X$ of size $x$.
For the left-hand side, this is clear: for a given map,
there are $x^v$ such mappings, where $v$ is the number of vertices of the map.
But we can count these pairs in another way.
Let us consider the pairs $(M,\varphi)$ for which the image set of $\varphi$
is a given set $X' \subseteq X$.
If $r$ is the size of $X'$, such a pair is the same thing as
a unicellular map colored with colors in $\llbracket 1,r\rrbracket$
and each color being used at least once.
Therefore, for a fixed $X'$, one has $A_r(n)$ such pairs $(M,\varphi)$.
As there are $\binom{x}{r}$ sets $X' \subseteq X$ of size $r$,
there are in total 
\[\sum_{r\geq 1} A_r(n) \binom{x}{r}\]
pairs $(M,\varphi)$, which proves identity~\eqref{HarerZagierCombi}.\medskip

Thus, it suffices to prove that $A_r(n)=(2n-1)!!\ \!2^{r-1}\binom{n}{r-1}$. 
Our main bijection sends unicellular maps colored
with colors in $\llbracket 1,r\rrbracket$           
(each color being used at least once) onto
\cdec trees with $n$ edges, where each (signed) cycle has a color in $\llbracket 1,r\rrbracket$,
and such that each color in $\llbracket 1,r\rrbracket$ is used by at least one cycle. 
Each of the $r$ colors yields a (non-empty) C-permutation, which
can be represented as a signed sequence, according to Lemma~\ref{lem:iso_signed}.
%GC2013: ici le referee n'aimait pas la formulation
% et j'avoue qu'à la relecture moi non-plus. alors j'ai changé
%Then one can concatenate these $r$ signed sequences into a unique sequence $S$ 
%of length $n+1$,  
Then one can encode these $r$ signed sequences
$^{\epsilon_1}S_1,^{\epsilon_2}S_2\dots,^{\epsilon_r}S_r$ by the triple
$(S,T,U)$ where $S=S_1S_2\dots S_r$ is their concatenation (it is a sequence of length $n+1$),
where $T=(\epsilon_1,\epsilon_2,\dots,\epsilon_r)$ is the $r$-tuple giving their signs and where $U$
is the subset of $r-1$ elements among the elements between positions $2$ and
$(n+1)$ in $S$, that indicates the starting elements of the sequences
$S_2,\dots, S_r$.
%(in order to recover from $S$ the $r$ signed sequences). 
%GC 2013 :added
For instance if $r=3$ and if the signed sequences corresponding respectively
to colors $1,2,3$ are $S_1=^+(3,9,4)$, $S_2=^-(5,8,6,2)$, and $S_3=^-(1,7)$, 
then the concatenated sequence is $S=(3,9,4,5,8,6,2,1,7)$,
together with the $3$ signs $T=(+,-,-)$ and the two selected elements $U=\{5,1\}$. 
It is clear that this correspondence is bijective.
Hence the number of such \cdec trees is $(n+1)!\ \!2^r\binom{n}{r-1}$, and 
by Theorem~\ref{theo:main},
% GC:ici un hack pourri pour permettre au carré de fin de preuve de se coller à
% droite de l'équation (avec double-dollar il se met en dessous, à la ligne):
%
\\
$~\hspace{2cm}\displaystyle
A_r(n)=2^{-n-1}\Cat(n)(n+1)!\ \!2^r\binom{n}{r-1}=(2n-1)!!\ \!2^{r-1}\binom{n}{r-1}.
$
\end{proof}\medskip

% GC2: modified next sentence:
%The summation formulas of papers \cite{ScVa08,MoVa09}
%(these are formulas for colored bipartite maps,
%taking the number of colors or the degree distributions into account)
%can also be recovered with our bijection.
%The proofs follow roughly the same guideline,
%but are omitted here for brevity.
%The papers \cite{ScVa08,MoVa09} contain other summation formulas, that deal
%with colored bipartite maps, taking the number of colors or the degree
%distributions into account.
%They can all be recovered from our bijection.
%The proofs follow roughly the same guideline,
%but are omitted here for brevity.

%VF: To make things shorter, I suggest not to state explicitely
%``Jackson's summation formula''
%VF5:j'ai décommenter tout ça pour la version longue.

%GC7: même faux titre que ci-dessus, je l'ai également viré. Apparemment il y a
%eu un petit problème d'édition. Zut, va falloir relire!
\subsubsection{A summation formula for bipartite unicellular maps.}

By Theorem~\ref{theo:main}, a \cdec tree associated to a 
bipartite unicellular map is a bipartite plane tree 
 such that each signed cycle must contain only white (resp. black) vertices. 
Recall that the $n+1$ vertices carry distinct labels from $1$ to $n+1$
(the ordering follows by convention a left-to-right depth-first traversal, 
see Figure~\ref{fig:C-tree}(c)). 
Without loss of information the $i$ black vertices (resp. $j$ white vertices)
can be relabelled from $1$ to $i$ (resp. from $1$ to $j$) in the order-preserving
way; we take here this convention for labelling the vertices of such a \cdec tree. 
We now recover the following summation formula due to Jackson 
%EF6: ai rajoute les references aux autres preuves bijectives connues)
(different bijective proofs have been given in~\cite{ScVa08} and in~\cite{BeHa}): 

\begin{prop}[Jackson's summation formula~\cite{Ja88}]\label{prop:Ja}
Let $B(v,w;n)$ be the number of bipartite unicellular maps with $n$ edges, $v$ black vertices
and $w$ white vertices. 
Then for $n\geq 1$
$$
\sum_{v,w}B(v,w;n)y^vz^w=n!\sum_{r,s\geq 1}\binom{n-1}{r-1,s-1}\binom{y}{r}\binom{z}{s}.
$$
\end{prop}

%\begin{proof} Omitted. The proof is an extension of the one of
%Proposition~\ref{prop:hazaA1} to the bipartite case. It makes use of the
%classical fact
%that the number of bipartite plane trees with $n$ edges, $i$ black vertices and
%$j$ white vertices, is given by the Narayana number $\frac1{n}\binom{n}{i}\binom{n}{j}$.
%\end{proof}
\begin{proof}
    As for the Harer-Zagier formula, there is a well-known combinatorial 
    reformulation of this statement.
    Namely, it suffices to prove that, for $r,s\geq 1$,
 the number $B_{r,s}(n)$ of bipartite unicellular maps with $n$
edges, each black (resp. white) vertex having a so-called \emph{b-color} in $\llbracket1,r\rrbracket$ 
(resp. a so-called \emph{w-color} in $\llbracket1,s\rrbracket$),  
such that each b-color in $\llbracket1,r\rrbracket$ (resp. w-color
in $\llbracket1,s\rrbracket$) is used at least once, is given by $B_{r,s}(n)=n!\binom{n-1}{r-1,s-1}$. 
For $n,i,j$ such that $i+j=n+1$, 
consider a bipartite \cdec tree with $n$ edges, $i$ black vertices, $j$ white vertices, where each black (resp. white) signed cycle  has a b-color
in $\llbracket1,r\rrbracket$ (resp. a w-color in $\llbracket1,s\rrbracket$), and each b-color in $\llbracket1,r\rrbracket$ (resp. w-color
in $\llbracket1,s\rrbracket$) is used at least once. By the same argument as in Proposition~\ref{prop:hazaA1},
the C-permutation and b-colors on black vertices can be encoded by a sequence $S_b$ 
of length $i$ of distinct integers in $\llbracket1,i\rrbracket$, together with
a sequence of $r$ signs and 
a subset of $r-1$ elements among the $i-1$ elements at positions from $2$ to $i$ in $S_b$. And the C-permutation and w-colors on white vertices can be encoded by a sequence $S_w$ 
of length $j$ of distinct integers in $\llbracket1,j\rrbracket$, together with
a sequence of $s$ signs and 
a subset of $s-1$ elements among the $j-1$ elements at positions from $2$ to $j$ in $S_w$. Hence there are $\Nar(i,j;n)2^{r+s}i!\binom{i-1}{r-1}j!\binom{j-1}{s-1}$
such \cdec trees, where $\mathrm{Nar}(i,j;n)$ (called the \emph{Narayana} number) is the  number of bipartite plane trees with $n$ edges, $i$ black vertices and $j$ white vertices, given by 
$\mathrm{Nar}(i,j;n)=\frac1{n}\binom{n}{i}\binom{n}{j}$.  By Theorem~\ref{theo:main},
\begin{align*}
    B_{r,s}(n)&=&2^{-n-1}2^{r+s}\sum_{i+j=n+1}\Nar(i,j;n)i!j!\binom{i-1}{r-1}\binom{j-1}{s-1}\\
&=&n!(n-1)!\frac{2^{r+s-n-1}}{(r-1)!(s-1)!}\sum_{i+j=n+1 \atop i \ge r, j \ge s}\frac{1}{(i-r)!(j-s)!}.
\end{align*}
But we have
$$
\sum_{i+j=n+1}\frac{1}{(i-r)!(j-s)!}=\sum_{i+j=n+1-r-s}\frac{1}{i!j!}=\frac{2^{n+1-r-s}}{(n+1-r-s)!}.
$$
Hence $B_{r,s}(n)=n!\binom{n-1}{r-1,s-1}$. 
\end{proof}\medskip

%GC2013: ajout de la section de Valentin (Harer-Zagier par degrés)
\subsubsection{A refinement of the Harer-Zagier summation formula}
\label{subsec:Refinement_Harer_Zagier}
The proof method above can be used to keep track of the vertex
degree distribution in the Harer-Zagier formula.
Before stating the resulting formula, let us mention that 
other methods can also keep track of this statistics,
for instance the bijective approach developed in \cite{BeHa}\footnote{
Proposition~\ref{prop:Refinement_Harer_Zagier} can also be deduced
from a formula of A. Morales and E. Vassilieva (Proposition~\ref{ThmMoralesVassilieva} below)
using \cite[Lemma 9]{BernardiEtAlSeparation}.}.
Thus, although the formula had not yet been stated explicitly in the literature
(as far as we know), all the elements needed to prove it were already
there\footnote{This formula was known to an anonymous referee, who suggested we
include its proof in the present paper.}.
Of course, the proof presented here is new and fits in our unified framework.

\begin{prop}
    \label{prop:Refinement_Harer_Zagier}
    Let $m_\rho$ and $p_\lambda$ be the monomial and power sum bases of the ring of
    symmetric functions and $\xx$ be an infinite set of variables.
    We denote by $\Ai(\lambda)$ the number of unicellular maps with
    degree distribution $\lambda$.
    Then, for any integer $n$, one has:
    \[ \sum_{\lambda \vdash 2n} \Ai(\lambda) p_\lambda(\xx)
    = \sum_{\rho \vdash 2n} \frac{n (2n-\ell(\rho))!}
    {(n-\ell(\rho)+1)!} 2^{\ell(\rho)-n} m_\rho(\xx).\]
\end{prop}
Let us make three remarks on this statement.
First, together with the trivial fact that $\Ai(\lambda)=0$ if
$\lambda$ is a partition of an odd number,
it entirely determines the numbers $\Ai(\lambda)$
(as power sums form a basis of the ring of symmetric functions).
Second, it implies the Harer-Zagier summation formula
(which can be recovered by setting $\xx=(1,\dots,1,0,\dots)$
with exactly $x$ times the value $1$).
Third, it admits an equivalent combinatorial formulation,
that we shall present now.

As in the previous subsection, we shall consider colored maps,
that is maps whose vertices are colored with numbers from $1$ to $r$,
each color being used at least one.
For such a map, one can consider its colored vertex degree distribution:
by definition, it is the composition $\II=(I_1,\dots,I_r)$ such that $I_k$
is the sum of the degrees of the vertices of color $k$.

We denote by $\AiC(\II)$ the number of colored maps with
colored vertex degree distribution $\II$.
Then, using the tools of \cite[Section 2]{MoVa09},
one can easily show that Proposition \ref{prop:Refinement_Harer_Zagier}
is equivalent to the following statement,
that we can prove using our main bijection.
\begin{prop}
    \label{prop:Formule_AiC}
    For any composition $\II$ of $2n$ of length $r$,
    \[\AiC(\II)=\frac{n (2n-r)!}
        {(n-r+1)!} 2^{r-n}.\]
\end{prop}
\begin{proof}
   We shall first consider the case where $\II$ has length $n+1$.
    This means that we count rooted unicellular maps with $n$ edges and
    vertices of $n+1$ colors.
    But a unicellular map is necessarily connected
    and, hence, has at most $n+1$ vertices.
    Therefore, we are counting rooted unicellular maps with $n+1$ labelled
    vertices, that is, rooted plane trees with labelled vertices of
prescribed degrees. 
%EFMay ajout de references 
In that case, one can show that $\AiC(\II)=2n!$ by 
several methods (e.g., the cycle lemma, or 
Pitman's aggregation process~\cite{Pitman}, 
or the more recent method by Bernardi and Morales~\cite{BernardiMoralesSymet}). We give
here a short proof by induction.%  
    Note that an unrooted vertex-labelled plane tree cannot have any symmetry
    and thus can always be rooted in $2n$ ways.
    So we shall rather compute the number $\UT(\II)$
    of unrooted vertex-labelled plane trees with degree
    distribution $\II$.

    We shall prove by induction that $\UT(\II)=(n-1)!$.

    For $n=1$, the only possibility is $\II=(1,1)$ and there is
    only one tree with such degree distribution: 
    $\begin{tikzpicture}[scale=.5]
    \tikzstyle{vertex}=[circle,draw,inner sep=0.5pt,minimum size=1mm]
    \node (v1) at (0,0) [vertex] {\footnotesize $1$};
    \node (v2) at (1,0) [vertex] {\footnotesize $2$};
    \draw (v1)--(v2);
    \end{tikzpicture}$.
    Thus $\UT( (1,1))$.

    Let $\II$ be a composition of $2n$ of length $n+1$.
    This composition must contain a part equal to $1$
    and without loss of generality ($\UT(\II)$ is invariant by permutation
    of the parts of $\II$), we may assume that $I_{n+1}=1$.
    This means that we are counting trees $T$,
    in which the vertex labelled $n+1$ is a leaf.
    Denote by $j$ the label of the vertex to which this leaf is attached.
    Then, removing the leaf $n+1$ from $T$, we obtained
    a tree $T'$ of degree distribution 
    $\II^{(j)}=(I_1,\dots,I_j-1,\dots,I_n)$.
    For each such tree $T'$, a new leaf labelled $n+1$ can be attached
    to the vertex $j$ in $I_j-1$ ways (recall that we are dealing with plane
    trees).
    Therefore,
    \[\UT(\II)= \sum_{j=1}^n (I_j-1) \cdot \UT(\II^{(j)}).\]
    From this induction relation, it is immediate to see that 
    $\UT(\II)=(n-1)!$ for any composition $\II$ of $2n$ of length $n+1$.
    Considering rooted trees instead  of unrooted trees,
    we get that, in this case
    \[\AiC(\II)=2n! .\]

    Turn back to the general case. Let $\II$ be a composition of $2n$
    of length smaller than $n+1$.
    Let us consider a composition $\HH$ of length $n+1$ refining $\II$.
 and consider a labelled plane tree $T$ whose vertex degree distribution is $\HH$
    (because of the labels, this distribution is ordered and can be
    seen as a composition).
     
Let us color our tree as follows:
we give color $s$ to a vertex if 
the corresponding part of $\HH$ is contained in the $s$-th part of $I$. 
Since the tree is labelled, 
vertices with the same color are totally ordered.
Hence if we add the data of a sign per color ($2^r$ choices
for all signs), using Lemma~\ref{lem:iso_signed},
we can see the vertices with the same color as endowed
with a $C$-permutation.

Putting all these $C$-permutations together, we obtain a $C$-permutation
of the vertices of the tree $T$, which has the following property:
vertices in the same cycle always have the same color.
Applying our main bijection (Theorem~\ref{theo:main}),
we obtain a unicellular map.
The vertices of this map have a canonical coloration,
as each vertex corresponds to a cycle of the $C$-permutation.
By construction, this colored map has colored degree distribution $\II$. 

To sum up, by Theorem~\ref{theo:main} and the construction above,  
 each colored unicellular map with colored 
degree distribution $\II$ can be obtained in $2^{n+1}$
different ways from
\begin{itemize}
    \item a labelled plane tree $T$ of vertex
     degree distribution given by $\HH$  
     for {\em some} refinement $\HH$ of $\II$ of length $n+1$;
    \item the assignment of a sign to each color.
\end{itemize}
The number of possible signs is always $2^r$,
so this yields a constant factor.
For a given composition $\HH$, the number of corresponding 
trees is $2n!$ (as seen above);
in particular, it does not depend on $\HH$.
Besides, looking at compositions as descent sets,
it is easy to see that there are 
\[\binom{2n-\ell(\II)}{n+1-\ell(\II)}\]
refinements $\HH$ of $\II$ of length $n+1$.
Finally, by Theorem~\ref{theo:main}, we get:
\[
2^{n+1} \AiC(\II)=2^r \cdot \binom{2n-\ell(\II)}{n+1-\ell(\II)}
\cdot 2n!,
\]
which simplifies to the claimed formula.
\end{proof}

%EFMay du coup j'enleve cette remarque
%\begin{remark}
%The {\em tree} case ($\ell(\II)=n+1$) in this proposition is certainly well-%known.
%In particular, it is equivalent to the refined enumeration formula counting Cayley trees by vertex degrees~\cite{CayleyTrees}.
%\end{remark}

%VF5:ajout de la preuve de MV
%GC7: A tiens encore un titre. SI l'on excepte les deux titres bizarres que
%j'ai déjà effacés, c'est le seul dans sa catégorie. Bon, allez, je le vire
%aussi.
\subsubsection{A refinement.}

A. Morales and E. Vassilieva \cite{MoVa09} have established a
very elegant summation formula for bipartite maps, counted with respect
to their degree distributions,
%EF6 ai rajoute cette remarque pour justifier que c'est bien un refinement
%GC7: puisque j'ai viré le titre ``refinement'' je le met ici:
which can be viewed as a refinement of Jackson's summation formula
(indeed, it is an easy exercise
to recover Jackson's summation formula out of it). 
%EFMAy ajout de la phrase suivante qui ne mange pas de pain
It can be noted that this formula is to Jackson's summation formula 
what Proposition~\ref{prop:Refinement_Harer_Zagier}
is to the Harer-Zagier summation formula: 

\newcommand{\citationProblematiqueBis}{\cite[Theorem 1]{MoVa09}}
\begin{prop}[Morales and Vassilieva \citationProblematiqueBis]
    Let $m_\lambda$ and $p_\rho$ be the monomial and power sum bases of the ring of
    symmetric functions and $\xx$ and $\yy$ two infinite sets of variables.
    Then, for any $n\geq 1$,
    %EF6 j'ai echange p et m, est-ce que c'est correct ?
    \[ \sum_{\lambda,\mu \vdash n} \Bi(\lambda,\mu) p_\lambda(\xx) p_\mu(\yy)
    = \sum_{\rho,\nu \vdash n} 
    \frac{n (n-\ell(\rho))! (n-\ell(\nu))!}{(n+1-\ell(\rho)-\ell(\nu))!}
    m_\rho(\xx) m_\nu(\yy).\]
    \label{ThmMoralesVassilieva}
\end{prop}
The original proof given in \cite{MoVa09} goes through a complicated bijection with
newly introduced objects called {\em thorn trees} by the authors. 
%EF6 ai rajoute reference a Olivier dont l'approche permet de trouver cette formule,
The bijective method in~\cite{BeHa} (which is well adapted to summation formulas)
also makes it possible to get the formula.  
And a short non-bijective proof has been given recently in 
\cite{KatyaSummationFormulaCharacters}
using characters of the symmetric groups and Schur functions.
We explain here how this result can be recovered from our bijection.
The proof is very similar to the one of Goupil-Schaeffer's formula.

Let us first recall that,
%GC2013: modif Valentin
as Proposition~\ref{prop:Refinement_Harer_Zagier},
 Proposition \ref{ThmMoralesVassilieva} can be reformulated
in purely combinatorial terms (without symmetric functions) using colored maps.

By definition here, a bipartite unicellular map is \emph{colored} 
by associating to each white (resp. black)
vertex a color in $\llbracket1, \ell_w\rrbracket$ (resp. $\llbracket1,
\ell_b\rrbracket$),
each color between $1$ and $\ell_w$ (resp. $\ell_b$) being chosen at least once
(note: we always think of the color $r$ of a white vertex as
{\em different} from the color $r$ of a black vertex).
To a colored bipartite map with $n$ edges one can associate its 
%EF6 ai change block degree and colored degree, ai un peu reformule
\emph{colored degree distribution}, 
that is, the pair $(\II,\JJ)$ of compositions of $n$, where the 
 $k$-th part of $\II$ (resp. of $\JJ$) is the sum of the degrees of the white (resp. black)
vertices with color $k$.

We denote by $\BiC(\II,\JJ)$ the number of colored bipartite unicellular maps
of colored degree distribution $(\II,\JJ)$.
Then Proposition \ref{ThmMoralesVassilieva} is equivalent to the following statement
\cite[paragraph 2.4]{MoVa09}:

\begin{prop}\label{PropReformulationMV}
    For any compositions $\II$ and $\JJ$ of the same integer $n$ 
    which satisfy $\ell(\II)+\ell(\JJ) \leq n+1$, one has
    \[\BiC(\II,\JJ)=\frac{n (n-\ell(\II))! (n-\ell(\JJ))!}{(n+1-\ell(\II)-\ell(\JJ))!}.\]
\end{prop}

\begin{proof}
%GC2013: modif Valentin
The proof of this proposition is very similar to the one of
 Proposition~\ref{prop:Formule_AiC}.
%EFMay ajout cas l+m=n+1 comme premier cas traite
First, in the case where $\ell(\II)+\ell(\JJ)=n+1$, 
a proof by induction similar
to the one in Proposition~\ref{prop:Formule_AiC} yields 
$$
\BiC(\II,\JJ)=n(\ell(\II)-1)!(\ell(\JJ)-1)!=n(n-\ell(\II))!(n-\ell(\JJ))!.
$$  
%    Note that, by definition, $\ell(\II)+\ell(\JJ)$ is at most the number
%    of vertices of the map, so it must be smaller or equal to $n+1$
We now turn to the general case. 
Let us consider two compositions $\HH$ and $\KK$ which refine respectively
    $\II$ and $\JJ$, and such that $\ell(\HH)+\ell(\KK)=n+1$.
    %EF6 ajoute cette definition utile dans la suite
    A part of $\HH$ (resp. of $\KK$) is said to have color $r$ if it is contained 
    in the $r$-th part of $\II$ (resp. of $\JJ$). 
    %    We denote $h_1=\ell(\HH)-\ell(\II) \geq 0$, $h_2=\ell(\KK)-\ell(\JJ)$
%    and $h=h_1+h_2$.
%    By definition, $h=n+1 - \ell(\II)-\ell(\JJ)$ depends only on $\II$ and $\JJ$.

    Consider a labelled %EF6 change bicolored en bipartite, reformule un peu
    bipartite tree $T$ whose white vertex degrees (in the order given by the labels)
    follow the composition $\HH$ and whose black vertex degrees follow
    the composition $\KK$. 
    % EF6: ici j'ai un peu change, en effet le 2p_r+1 me parait curieux...
%    As in the proof of Proposition~\ref{prop:GS},
%    the vertices of $T$ are canonically labelled by the set 
%$$ \{w_r^t; 1 \leq r \leq \ell, 1 \leq t \leq 2 p_r+1\}
%\sqcup \{b_r^t; 1 \leq r \leq m, 1 \leq t \leq 2 q_r+1\}.
%$$
%We think to the vertex $w_r^t$ and $b_r^t$ as colored with color $r$.
A black (resp. white) vertex is said to have color $r$ if 
the corresponding part of $\HH$ (resp. of $\KK$) has color $r$. 
Since the tree is labelled, 
the white (resp. black) vertices with the same color $r$ are totally ordered.
%GC2013: modi Valentin
% as $(w_r^1,w_r^2,\dots)$ (resp. $(b_r^1,b_r^2,\dots)$).
Hence if we add the data of a sign per color ($2^{\ell(I)+\ell(J)}$ choices
for all signs), using Lemma~\ref{lem:iso_signed},
we can see the vertices with the same color as endowed
with a $C$-permutation.

Putting all these $C$-permutations together, we obtain a $C$-permutation
of the vertices of the tree $T$, which has the following property:
the vertices in the same cycle always have the same color.
Applying our main bijection (Theorem~\ref{theo:main}),
we obtain a bipartite unicellular map.
The vertices of this map have a canonical coloration,
as each vertex corresponds to a cycle of the $C$-permutation.
By construction, this colored map has colored degree distribution $(\II,\JJ)$. 

%EF6: remplace "Conversely" par "To sum up" en effet je n'ai pas
% l'impression qu'on definit ici une construction inverse mais plutot qu'on
% recapitule ce qui vient d'etre explique
To sum up, by Theorem~\ref{theo:main} and the construction above,  
 each colored bipartite unicellular map with colored 
degree distribution $(\II,\JJ)$ can be obtained in $2^{n+1}$
different ways from
\begin{itemize}
    \item a labelled bipartite tree $T$ of white (resp. black) vertex
     degree given by $\HH$ (resp. $\KK$) 
     for {\em some} refinements $\HH$ and $\KK$  with
$\ell(\HH)+\ell(\KK)=n+1$;
    \item the assignment of a sign to each color.
\end{itemize}
The number of possible signs is always $2^{\ell(I)+\ell(J)}$,
so this yields a constant factor.
For given compositions $\HH$ and $\KK$, the number of corresponding 
trees is
\[n (\ell(\HH) -1)! (\ell(\KK) -1)! \]
Thus we have to count the number of refinements $\HH$ (resp. $\KK$) of 
$\II$ (resp. $\JJ$) with a given value $\ell$ of $\ell(\HH)$ 
(resp. $m$ of $\ell(\KK)$).
It is easily seen to be equal to
\[\binom{n-\ell(\II)}{\ell-\ell(\II)} \text{ (resp. } 
\binom{n-\ell(\JJ)}{m-\ell(\JJ)} \text{ ).}\]
Finally, by Theorem~\ref{theo:main}, we get:
\[
    2^{n+1} \BiC(\II,\JJ)=2^{\ell(\II)+\ell(\JJ)} \sum_{\ell+m=n+1 \atop \ell \geq \ell(\II),
    \ m \geq \ell(\JJ)} 
    n  (\ell-1)! (m-1)!  \binom{n-\ell(\II)}{\ell-\ell(\II)}
     \binom{n-\ell(\JJ)}{m-\ell(\JJ)}.
\]
Denoting $h=n+1-\ell(\II)-\ell(\JJ)$ and setting $h_1=\ell - \ell(\II), 
h_2=m-\ell(\JJ)$
in the summation index, 
%GC7: style amsart oblige, toutes les équations ci-dessus dépassent de la
%ligne...
% Je réécris donc un peu:
%the equation above writes as
the right-hand side of the previous equation writes as:
\[\hspace{-.5cm} 2^{n+1} \BiC(\II,\JJ)
    = 
2^{\ell(I)+\ell(J)}\!\! \sum_{h_1 +h_2 = h}\!\!
    n  (\ell(\II)+ h_1-1)! (\ell(\JJ)+h_2-1)!  \binom{n-\ell(\II)}{h_1}
     \binom{n-\ell(\JJ)}{h_2}.\]
But the relation $\ell(\II)+\ell(\JJ)+h_1+h_2=n+1$ implies that
%GC7: Les deux équations suivantes ne me paraissant pas assez cruciales pour
%être mises sur deux lignes (et vu que sur une seule lignes elle dépassaient le
%style amsart, je les ai simplement mises en mode texte:
% (i.e. passé de $$ à $ )
$(\ell(\JJ)+h_2-1)!  \binom{n-\ell(\II)}{h_1} = 
\frac{(n-\ell(\II))!}{h_1!}$
and
$
(\ell(\II)+ h_1-1)!  \binom{n-\ell(\JJ)}{h_2}=
\frac{(n-\ell(\JJ))!}{h_2!}.$
Plugging this in the 
%GC7: equation 
expression
above, we get
\begin{multline*}
    2^{n+1} \BiC(\II,\JJ)=2^{\ell(I)+\ell(J)} \cdot 
n \cdot (n-\ell(\II))! \cdot (n-\ell(\JJ))!
\sum_{h_1 + h_2 = h} \frac{1}{h_1!h_2!} \\
=2^{\ell(I)+\ell(J)} \cdot 
n \cdot (n-\ell(\II))! \cdot (n-\ell(\JJ))! \frac{2^h}{h!}.
\end{multline*}
The powers of $2$ cancel each other and we get the desired result.
\end{proof}

%GC7: ajouté cette sous-section:
\subsection{Covered maps, shuffles, and an identity
of~\cite{BeCha:covered}.}
\emph{Covered maps} were introduced in~\cite{BeCha:covered} as an
extension of the notion of tree-rooted map (map equipped with a spanning tree).
A covered map of genus $g$ is a rooted map $M$ of genus $g$, not necessarily
unicellular, equipped
with a distinguished connected subgraph $S$ (with the same vertex set as $M$) having the following
property:
\begin{quote}\it
viewed as a map, $S$ is a unicellular map, possibly of a
different genus than $M$.
\end{quote}
Here, in order to view $S$ ``as a map'', we equip it with the
map structure induced by $M$: the clockwise ordering of half-edges of $S$
around each vertex is defined as the restriction of the clockwise ordering
in $M$ (see~\cite{BeCha:covered}
for details). The genus $g_1$ of $S$ is an element of $\llbracket 0,
g\rrbracket$. For example, $g_1=0$ if and only if $S$ is a spanning tree of
$M$. In general, we say that the covered map $(M,S)$ has \emph{type} $(g,g_1)$.
%We now review some elementary properties of covered maps
%from~\cite{BeCha:covered}. Given a covered map $(M,S)$, the \emph{dual} covered map is defined as
%$(M^*,S^*)$, where $M^*$ is the dual map of $M$, and $S^*$ is the set of edges
%of $M^*$ which are not the dual of an edge of $S$. This construction
%generalizes the classical planar construction of \emph{dual spanning tree} to any
%orientable surface. If $M$, $S$, and $S^*$ have respective genus $g$, $g'$, and
%$g''$ then $g'+g''=g$. 
%Conversely, given two unicellular maps $S$ and $S'$, one can always construct a
%covered map $(M,S)$ of dual map $(M^*,S')$. Moreover, if $S$ and $S'$ have respectively
%$n_1$ and $n_2$ edges, an elementary result proved in~\cite{BeCha:covered}
%asserts that there are exactly
%${2(n_1+n_2)\choose n_1}$ ways to do so. 

Covered maps have an interesting duality property that generalizes the existence
of dual spanning trees in the planar case: namely, each covered map $(M,S)$ of
type $(g,g_1)$ has a dual covered map $(M^*, S')$ of type $(g,g_2)$ with
$g_1+g_2=g$. 
%By considering simultaneously the submap $S$ and its dual submap
%$S'$, and 
By extending ideas of Mullin~\cite{Mullin}, it is not difficult to describe the
covered map $M$ as a ``shuffle'' of the two unicellular maps $S$ and $S'$,
see~\cite{BeCha:covered}. It follows that the number $\mathrm{Cov}_{g_1,g_2}(n)$ of
covered maps of type $(g_1+g_2,g_1)$ with $n$
edges can be expressed as the following shuffle-sum \cite[eq. (6)]{BeCha:covered}:
 \begin{align}\label{eq:covered-first}
   \mathrm{Cov}_{g_1,g_2}(n) = \sum_{n_1+n_2=n} {2n\choose 2n_1}
   \epsilon_{g_1}(n_1) \epsilon_{g_2}(n_2).
 \end{align}
In the case $g_1=g_2=0$, this sum simplifies thanks to the Chu-Vandermonde identity,
and we have the remarkable result due to Mullin \cite{Mullin} (see \cite{Bernardi-TR} for a bijective
proof):
 \begin{align}\label{eq:mullin}
   \mathrm{Cov}_{0,0}(n) = % \sum_{n_1+n_2=n} {2n\choose 2n_1}
   %\mathrm{Cat}(n_1) \mathrm{Cat}(n_2) =
   \mathrm{Cat}(n)\mathrm{Cat}(n+1).
 \end{align}
 The main enumerative result of the paper~\cite{BeCha:covered}
 is a generalisation of \eqref{eq:mullin} to any genus, obtained via a
 difficult bijection:
 \begin{prop}[Bernardi and Chapuy, \cite{BeCha:covered}]\label{thm:covered}
   For all $n\geq1$ and $g\geq 0$, the number $\mathrm{Cov}_g(n)=\sum\limits_{g_1+g_2=g}\mathrm{Cov}_{g_1,g_2}(n)$ of covered maps of genus $g$ with
   $n$ edges is equal to:
 \begin{align*}
   \mathrm{Cov}_g(n) = \mathrm{Cat}(n)\mathrm{Bip}_g(n+1),
 \end{align*}
 where $\mathrm{Bip}_g(n+1)$ is the number of 
%EFMay enleve rooted, sinon on va croire que c'est unrooted si non precise 
bipartite unicellular maps
 of genus $g$ with $n+1$ edges. Equivalently, the following identity holds:
 \begin{align}\label{eq:covered-id}
   \sum_{g_1+g_2=g} \sum_{n_1+n_2=n} {2n\choose 2n_1}
   \epsilon_{g_1}(n_1) \epsilon_{g_2}(n_2)
   = \mathrm{Cat}(n)\mathrm{Bip}_g(n+1).
 \end{align}
 \end{prop}
\begin{proof}
 We denote as before by $c_g(m)$ the number of C-permutations of genus $g$ of a
 set of $m$ elements. By our main result, Theorem~\ref{theo:main}, the left-hand side of
 \eqref{eq:covered-id} can be rewritten as:
 $$  2^{-n-2}\sum_{g_1+g_2=g} \sum_{n_1+n_2=n} {2n\choose 2n_1}
 c_{g_1}(n_1+1) c_{g_2}(n_2+1) \mathrm{Cat}(n_1) \mathrm{Cat}(n_2).
  $$
  We now observe that:
  $${2n\choose 2n_1}\mathrm{Cat}(n_1)\mathrm{Cat}(n_2) =
  \mathrm{Cat}(n)\mathrm{Nar}(n_1+1,n_2+1;n+1)$$
  where as before the Narayana number $\mathrm{Nar}(i,j;n)$ is
  the number of bipartite plane trees with $n$ edges, $i$ black vertices and $j$
  white vertices (this last equality follows directly from the explicit
  expressions of Catalan and Narayana numbers; an interpretation in terms of
  planar tree-rooted maps is given by 
  the bijection of \cite{Bernardi-TR}). 
%  $$\mathrm{Nar}(i,j;n)=\frac1{n}\binom{n}{i}\binom{n}{j}$$
 Therefore we have:
 $$\mathrm{Cov}_g(n)= 2^{-n-2} \mathrm{Cat}(n) \hspace{-2mm}\sum_{g_1+g_2=g} 
 \sum_{n_1+n_2=n} \hspace{-2mm}
    c_{g_1}(n_1+1)c_{g_2}(n_2+1) \mathrm{Nar}(n_1+1,n_2+1;n+1).$$
 Now, the double-sum in this equation %quantity
% $$
%\sum_{g_1+g_2=g} 
%   \sum_{n_1+n_2=n} 
%    c_{g_1}(n_1+1)c_{g_2}(n_2+1) \mathrm{Nar}(n_1+1,n_2+1;n+1)
%$$
 is equal to the number of  bipartite C-decorated trees (that is, bipartite trees equipped with a
 C-permutation of the vertices that stabilizes each color class) with $n+1$
 edges and genus $g$: indeed in the double-sum, the quantities $g_1$ and $n_1+1$
 can be
 interpreted respectively as the genus of the restriction of the
 $C$-permutation to black vertices of the tree, and as the number of black
 vertices in the tree. By our main result, Theorem~\ref{theo:main}, this
 double-sum is therefore equal to $2^{n+2} \mathrm{Bip}_g(n+1),$
% Summing this quantity $\mathrm{Bip}_g(n_1+1-2,n_2+1)$ over $n_1+n_2=n$, i.e.,
% summing over the possible numbers of
% black and white vertices in a bipartite plane map of genus $g$, we obtain the right-hand
% side of 
 which proves \eqref{eq:covered-id}.
 \end{proof}
 \noindent The proof above and \eqref{eq:covered-first} also show the following fact. Let $G_1$ be the
 genus of the submap $S$ in
 a covered map $(M,S)$ of genus $g$ with $n$ edges chosen uniformly at random,
 and let $G_{\circ}$ be the genus of the restriction to white vertices of the
 $C$-permutation in a bipartite $C$-decorated tree of genus $g$ with $n+1$
 edges chosen
 uniformly at random. Then the random variables
 $G_1$ and $G_{\circ}$ have the same distribution.

% Now it is easy to guess that,
 It is possible to prove that,
 when $g$ is fixed and $n$ tends to infinity, the variable 
 $G_{\circ}$ is close to a  binomial random variable $B(g,1/2)$: the idea behind this property is
 that a random bipartite tree with $n+1$ edges has about $n/2+O(\sqrt{n})$ vertices
 of each color with high probability, and that with high probability the $C$-permutation of its
 vertices is made of $g$ cycles of length $3$, that independently ``fall''
 into each of the two color classes with probability $1/2$. Giving a proper proof of these
 elementary statements would lead us too far from our main subject, so we
 leave to the reader the details of a proof along these lines of the following fact, which
 was proved in \cite{BeCha:covered}  with no combinatorial interpretation:
 \begin{prop}[\cite{BeCha:covered}]
   Let $g\geq g_1\geq0$. When $n$ tends to infinity, the probability that a
   covered map of genus $g$ with $n$ edges chosen uniformly at random has type
   $(g,g_1)$ tends to $\displaystyle2^{-g} {g \choose g_1}$.
 \end{prop}

% We now give the idea behind this proposition, without giving all the details
% regarding the asymptotics, which are out of the scope of the present paper.
% \begin{proof}[sketch] Let $T$ be a bipartite C-decorated tree of genus $g$ with
%   $n+1$ edges chosen uniformly at random (by \emph{bipartite}, we mean that
%   the C-permutation of its vertices stabilizes the canonical bicoloration),
%   and let $C_{\circ}$ be the restriction of the associated $C$-permutation to 
%   the white vertices of $T$. Then the proof of Theorem~\ref{thm:covered} shows that the genus
%   of $C_{\circ}$ has the same distribution as the genus of the submap $S$ in a random
%   covered map $(M,S)$ of genus $g$ and size $n$.
%
%  Now, we use two facts that we give without proof. First, a random
%  $C$-permutation of genus $g$ on a set of $n$ elements has $g$ cycles of
%  and $3$ and all other cycles of length $1$ with probability
%  $1-O(\frac{1}{n})$ when $n$ tends to infinity. Secondly, the
%  proportion of white vertices in a uniform
%  random bipartite plane tree with $n$ edges converges in law to $1/2$
%  when $n$ tends to infinity. From this two facts one easily deduces that with
%  probability $1-o(1)$, the $C$-permutation of the tree $T$ has $g$ cycles of length $3$
%
%
% \end{proof}

 To conclude this section, we mention that, in~\cite{BeCha:covered},  refined
 results were given that take more parameters
 into account ({\it e.g.}, the number of vertices and faces of the covered map).
 These extensions can be proved exactly in the same way as
 Proposition~\ref{thm:covered}, but we do not state them explicitly here, for the
 sake of brevity.

%GC2: title has to fit on one line, so changed ``application'' to ``link''
% not very satisfied with that
%GC3: ai tenté d'uniformiser ``Stanley character polynomials'' partout, suite à
%remarque de Valentin
%EF6 je me demande si ce ne serait pas plutot Stanley's character polynomials, comme vous voulez
\section{Computing Stanley character polynomials}
\label{SectStanley}
%VF5: j'ai repris un peu de place dans ce paragraphe pour donner un peu plus de
%détails et mettre les formules sur des lignes à part.
\subsection{Formulation of the problem}
% GC2: pourrait-on trouver une première phrase un peu plus motivante ?
% . . .  j'ai pas trouvé!
We now consider the following enumerative problem. 
For $n$ a fixed integer, 
we would like to compute the generating series
\[F_n(p_1,p_2,\cdots;q_1,q_2,\cdots)= \sum_{(M,\phi)} \wt(M,\phi)\]
of pairs $(M,\phi)$ where
% GC2: removed itemize for space reasons.
% EF ajout de with n edges, est-ce correct ?
$M$ is a rooted bipartite unicellular map with $n$ edges, and $\phi$ is a mapping from
the vertex set $V_M$ of $M$ to positive integers, satisfying
the following {\em order condition}:
\begin{quote}\it
for each edge $e$ of $M$, one has $ \phi(b_e) \geq \phi(w_e), $ where $b_e$ and
$w_e$ are respectively the black and white extremities of $e$. 
    \end{quote}
% GC2: rephrased definition of weight:
The weight of such 
a pair is
$\wt(M,\phi):=\prod_{v \in V_M^\circ} p_{\phi(v)}
\prod_{v\in V_M^\bullet} q_{\phi(v)},$ 
where $V_M^\bullet$ and $V_M^\circ$ are respectively the sets
of black (resp. white) vertices of~$M$.
%In other terms, each black (resp. white) vertex of image $i$ by $\phi$
%corresponds to a factor $q_i$ (resp. $p_i$) in the weight of $(M,\phi)$.

%The weight of such a couple is
%$\wt(M,\phi)=\prod_{i \geq 1} p_i^{|V_M^\circ \cap \phi^{-1}(i)|}
%\prod_{i \geq 1} q_i^{|V_M^\bullet \cap \phi^{-1}(i)|}, $\\
%where $V_M^\bullet$ and $V_M^\circ$ are respectively the sets
%of black (resp. white) vertices of $M$.
%In other terms, each black (resp. white) vertex of image $i$ by $\phi$
%corresponds to a factor $q_i$ (resp. $p_i$) in the weight of $(M,\phi)$.

Our motivation comes from representation theory of the symmetric group.
This topic is linked to map enumeration by the following formula 
conjectured in \cite{St} and proved in
 \cite{Fe}.
Let $\pp=p_1,\cdots,p_r$ and $\qq=q_1,\cdots,q_r$ be two finite lists of positive integers
of the same length.
Then the evaluation
 of the generating series considered above is equal to 
 \begin{equation}\label{EqStanleyFormula}
     F_n(p_1,\cdots,p_r,0,\cdots;q_1,\cdots,q_r,0,\cdots)
   = L(L-1)\cdots(L-n+1) \hat{\chi}^{\lambda}( (1\ 2\ \cdots\ n) ),
\end{equation}
where:
\begin{enumerate}%[topsep=0pt, partopsep=0pt, itemsep=0pt, parsep=0pt]
    \item[$\bullet$] $\lambda$ is the partition with $p_1$ parts equal to $q_1+\cdots+q_r$,
    $p_2$ parts equal to $q_2+\cdots+q_r$, and so on\dots
%       \[ \left( \underbrace{q_1+\cdots+q_r,\cdots,q_1+\cdots+q_r}_{p_1\text{ times}},
%      \underbrace{q_2+\cdots+q_r,\cdots,q_2+\cdots+q_r}_{p_2\text{ times}},\cdots\right); \]
     \item[$\bullet$] $L=\sum_{1 \leq i \leq j \leq r} p_i q_j$ is the number
of boxes of $\lambda$;
% EF ai change S_n pour S_L, est-ce correct ?
     \item[$\bullet$] $\hat{\chi}^{\lambda}$ is the normalized character of
      the irreducible representation of $S_L$ associated to $\lambda$;
     \item[$\bullet$] $(1\ 2\ \cdots\ n)$ is an $n$-th cycle seen as a permutation of $S_L$
         (if $n >L$, it is not defined but, as the numerical factor is $0$, it is not
         a problem).
\end{enumerate}

\begin{remark}
    In \cite{St,Fe}, this formula is stated under a slightly different form.
    We call $G_n$ the same generating series as $F_n$ except that the order
    condition is replaced by the following {\em maximum condition}:
    \begin{quote}\it
        for each black vertex $b$, one has $\phi(b)= \max \phi(w)$, where the
        maximum is taken over all white neighbours $w$ of $b$.
    \end{quote}
    Then the main theorem of \cite{Fe} states that
    \[G_n(p'_1,\cdots,p'_r,0,\cdots;q'_1,\cdots,q'_r,0,\cdots)
    =L(L-1)\cdots(L-n+1) \hat{\chi}^{\lambda}( (1\ 2\ \cdots\ n) ),\]
    where everything is defined as above except that
    \begin{quote}\it
        $\lambda$ is the partition with $p'_1$ parts equal to $q'_1$,
            $p'_2$ parts equal to $q'_2$, and so on\dots
    \end{quote}
    This result is clearly equivalent to \eqref{EqStanleyFormula} by setting:
    \[\forall i \ge 1, \begin{cases}
        p_i=p'_i \\
        q_i=q'_i-q'_{i+1}
    \end{cases}.\]
\end{remark}

\subsection{A new expression for $F_n$}
Our main bijection allows us to express the generating series $F_n$ in terms
of the corresponding generating series for plane trees:
\[R_{n+1} (\pp,\qq) = \sum_{(T,\phi)} \wt(T,\phi),\]
where the sum runs over all pairs $(T,\phi)$, $T$ being a plane tree 
%EFMay ajout with n edges
with $n$ edges 
and $\phi$ a function $V_T \to \NN$ satisfying the order condition.

The strange notation $R_{n+1}$ comes from the following fact:
A. Rattan has proved \cite{Ra07} that this generating series is
the $n+1$-th free cumulant $R_{n+1}$
of the transition measure of the Young diagram $\lambda$
($\lambda$ stands here for the Young diagram defined in terms on $\pp$ and $\qq$
in the previous paragraph).
%$R_{n+1}$ can be seen as a EF: remplace series par polynomial, est-ce correct ? 
% GC3: remis series
%series in $\pp$ and $\qq$).
Free cumulants have become in the last few years an important tool in (asymptotic)
representation theory of the symmetric groups, see for example the work of P. Biane
\cite{Biane}.

% EF donne l'expression explicite, legeres modifications
%EFMay ajout footnote
Let us define an operator $D$ by\footnote{The second equality is obtained
by similar arguments as in the proof of Proposition~\ref{prop:hazaA1}.}:  
%GC3 added := instead of =
\[ D(x^k) := \sum_{g \geq 0} c_g(k) x^{k-2g}=k!\sum_{r=1}^k2^r\binom{k-1}{r-1}\binom{x}{r}, \] 
% \text{ for }x=p_1,p_2,\cdots,q_1,q_2,\cdots,$
$D$ being extended multiplicatively to monomials in distinct variables, 
and then extended linearly to multivariate polynomials
%GC3 added
and series 
(in
particular,
%GC3 changed polynomials
series 
in the variables $\pp$ and $\qq$).
% GC2: added a remark on computation and explicitness
%EF du coup j'enleve ca Recall that $c_g(k)$ is an explicit quantity, so the action of
% $D$ is totally explicit and easily implemented.%GC2: changed proposition to theorem
\begin{theo}\label{thm:stanley}
    For any $n \geq 1$, 
%GC3 added    
    one has 
     $2^{n+1} F_n = D(R_{n+1}). $
\end{theo}
\begin{proof}
    A pair $(M,\phi)$ as above corresponds by the bijection of Theorem~\ref{theo:main} 
    to a bipartite $C$-decorated tree $T$, together with a function $\phi : V_T \to \NN$
    which fulfills the order condition and such that all vertices in a given cycle
    have the same image by $\phi$.
    %VF10: ajoute un brin d'explications
    Equivalently, we choose the tree $T$, a function $\phi : V_T \to \NN$ and then,
    for each $i \ge 1$ a $C$-permutation of 
    the white (resp. black) vertices of %EFMay plutot value que color $i$. 
value $i$. 
    The result follows directly.
\end{proof}

%GC3 added, Valentin peux-tu vérifier que ça fait sens ?
%VF oui
The free cumulant $R_{n+1}$ is the compositional
inverse of an explicit %formal power 
series \cite{Ra07}. Hence
Theorem~\ref{thm:stanley} gives an efficient,
easily implemented way of
computing %the expansion of 
Stanley character polynomials $F_n$.

\section{Counting 3-constellations}
\label{SectConstellations}
%The refined enumeration formula for bipartite maps due to A. Goupil and 
%G. Schaeffer (Proposition \ref{prop:GS}) has been further extended to
%the refined enumeration of {\em constellations} (see definition below)
%by D. Poulalhon and G. Schaeffer \cite{PoulalhonSchaefferConstellations}.
%The method in this article involves both algebra and combinatorics.
%
%Unfortunately, we have not been able to give a purely combinatorial proof
%of this fact using the construction of this article.
%We nevertheless present two results for $3$-constellations.
%
%Tout ceci est redit plus tard.. 
 
\subsection{Constellations and the Poulalhon-Schaeffer formula}
% GC8: j'ai réécrit tout cela, en prenant moins de risques sur l'aspect historique
% en plus je cite Cori-Machi, ça fera plaisir aux puristes!
%....
%Constellations have been introduced to encode transitive factorizations
%of a permutation into $m$ factors
%\alert{Est-ce que c'est vrai?}.
%In this sense, they generalize in a natural way bipartite maps,
%which correspond to the case $m=2$.
Constellations are a family of colored maps, depending on an integer parameter
$m\geq 2$, that are in bijection with
factorizations of a permutation into $m$ factors.
We refer to~\cite[Chapter~1]{ZvLa97} for a general discussion on constellations, and in
particular to Section 1.6.2 of this book for the correspondence between the factorization
viewpoint and the map-theoretic perspective that we adopt here
(see also \cite[Section~2]{BernardiMoralesFacto}, or the introduction
of~\cite{MBM-Schaeffer}).
For $m=2$, constellations are in bijection with bipartite maps, which are well-known
to be in bijection with factorizations of permutations into $2$~factors~\cite{CoriMachi}. 

\begin{defin}\label{def:constellation}
    An $m$-constellation is a map with circle and square vertices such that:
    \begin{enumerate}[label=(\roman*)]
        \item the circle vertices are colored with $m$ colors $1,2,\dots,m$;
        \item all edges have one circle and one square extremity;
        \item each square vertex is linked to exactly one circle vertex of each
            color;
        \item \label{ItemCyclic}
            moreover, the circle vertices around each square vertex appear
% GC8: ajouté counterclockwise (sinon en faisant le dessin j'étais pas sur de tourner
% dans le bon sens...)
counterclockwise
            in the cyclic order $1,2,\dots,m$.
    \end{enumerate}
    A constellation is {\em rooted} if we distinguish a corner of a circle
    vertex of color $1$.
    Unless mentioned explicitly, all constellations considered will be rooted.
 
    The size of an $m$-constellation is its number of square vertices.
\end{defin}

%\alert{Faire un dessin\dots}
%GC8: voici un dessin...
\begin{figure}[h!!!!!!]
\includegraphics[scale=0.9]{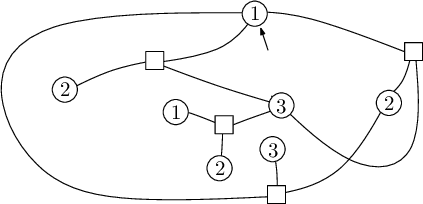}
  \caption{A rooted $3$-constellation of size $4$ (the root corner is pointed
with an arrow). This $3$-constellation has genus $1$ and is unicellular.}
\label{fig:3-constellation}
\end{figure}

%GC8: pour les références il me semble que Lando-Zvonkine est très détaillé et
%contient tout (dessin hypercartes vs. dessin cartes biparties, correspondance
%avec les factorisations, etc.) Peut-être faut-il aussi citer MBM-Schaeffer?
%\alert{Je manque de références. On devrait aussi signaler qu'il traine
%plusieurs définitions équivalentes dans la littérature.
%Est-ce que vous connaissez un papier où c'est expliqué?
%Et le lien avec les factorisations?}
%GC8 added
\noindent Note that several papers, {\it e.g.}~\cite{MBM-Schaeffer, BDFG:mobiles,
Chapuy:constellations, BernardiMoralesFacto}, 
use a different but equivalent definition of constellations
in terms of maps, where square vertices are replaced by ``black faces'' of
degree $m$. For
the purpose of using the bijection of Section~\ref{SubsectTrisections}, the
definition we use here will be much more convenient. 
%GC8: l'équivalence entre les deux représentations est tellement évidente qu'il
%ne me semble pas nécessaire d'en dire plus. Sinon on pourrait reciter ce qu'on
%a déjà cité plus haut comme cela:
%
%The equivalence of both is explained in \cite[Section 1.6.2]{ZvLa97} or
%\cite[Figure 2]{MBM-Schaeffer}.

Fix an $m$-constellation of size $n$.
%EF9 reecrit
The sequence of the degrees of its circle vertices of color $i$, taken in decreasing order, 
forms a partition $\lambda^{(i)}$ of size $n$.
The list $\lambda^{(1)},\dots,\lambda^{(m)}$ is called the {\em multitype}
of the constellation.

For unicellular $m$-constellations, 
the Euler formula links the
genus and the multitype 
(here, $\ell_i$ is the length of $\lambda^{(i)}$):
\[2g=n(m-1)+1-\sum_{i=1}^m \ell_i.\]

Using algebraic tools,
D. Poulalhon and G. Schaeffer have given a general formula for the number 
of unicellular $m$-constellations of size $n$
\cite[Theorem 1]{PoulalhonSchaefferConstellations}
with a given multitype. Though explicit, their formula requires quite heavy notations to be stated,
therefore we present here only the case $m=3$,
which is the only case we are able to attack with our combinatorial tools.

For a partition $\lambda$ of length $\ell$ and a non-negative integer $g$,
we denote
\begin{align*}
    a(\lambda)&= \prod m_i(\lambda)! ; \\
    S_g(\lambda)&= (\ell+2g-1)! \sum_{p_1+\dots+p_\ell=g}
    \prod_{i=1}^\ell \frac{1}{2p_i+1} \binom{\lambda_i-1}{2p_i}.
\end{align*}
\begin{theo}[Poulalhon and Schaeffer, 2002]
    Let $\lambda^{(1)}$, $\lambda^{(2)}$ and $\lambda^{(3)}$ be three partitions
    of lengths $\ell_1$, $\ell_2$ and $\ell_3$ and of the same size $n$, such that
    %%EF9 such that plutot que Suppose that 
    $$g=1/2 \cdot (2n+1-\ell_1-\ell_2-\ell_3)$$ 
    is a non-negative integer.
    Then the number $c_{\bm{\lambda}}$ of $3$-constellations of multitype 
    $(\lambda^{(1)}, \lambda^{(2)}, \lambda^{(3)})$ is given by the formula
    \begin{multline*}
        c_{\bm{\lambda}} = \frac{n^2}{2^{2g} a(\lambda^{(1)})
    a(\lambda^{(2)}) a(\lambda^{(3)})} \sum_{g_0 + g_1 +g_2 +g_3=g}
    \bigg[ (g_0!)^2 \binom{n-\ell_1-2g_1}{g_0} \\ 
    \cdot\binom{n-\ell_2-2g_2}{g_0} 
     \binom{n-\ell_3-2g_3}{g_0}
    S_{g_1}(\lambda^{(1)}) S_{g_2}(\lambda^{(2)}) S_{g_3}(\lambda^{(3)})\bigg].
    \end{multline*}
    \label{ThmPoulalhonSchaefferM3}
\end{theo}

%EF9 reecrit
The planar case ($g=0$) of this theorem: 
\begin{equation}
    \label{Eq3ConstPlanar}
                c_{\bm{\lambda}} = \frac{n^2}{a(\lambda^{(1)})     
           a(\lambda^{(2)}) a(\lambda^{(3)})} (\ell_1-1)! (\ell_2-1)! (\ell_3-1)!, 
\end{equation}
has been proved earlier by I. Goulden and D. Jackson
\cite[Theorem 3.2]{GJ92} and can be handled in a purely combinatorial
way \cite{BernardiMoralesFacto}. 

Finding a combinatorial proof of
Theorem \ref{ThmPoulalhonSchaefferM3} for higher genus
is an open problem (and seems difficult, due to the complexity of the formula).
We did not succeed in solving this problem but
we shall present two results in this direction.

\subsection{Refined enumeration of quasi-constellations}
Our bijection preserves the underlying multi-graph of a unicellular map,
but not the cyclic order around vertices.
Therefore the last condition in the definition of constellations is hard to
handle with our method.

This paragraph is devoted to the proof of a refined enumeration formula
for new objects 
that we call {\em $3$-quasi-constellations}, which are defined by the same
conditions as constellations, except that condition \ref{ItemCyclic} is dropped.

The formula obtained is surprisingly close to the one for constellations,
although we are not able to explain this phenomenon.

\def\tc{\tilde{c}}
\def\hc{\hat{c}}

\begin{prop}
    Let $\lambda^{(1)}$, $\lambda^{(2)}$ and $\lambda^{(3)}$ be three partitions
    of lengths $\ell_1$, $\ell_2$ and $\ell_3$ and of the same size $n$, 
    such that $g=1/2 \cdot (2n+1-\ell_1-\ell_2-\ell_3)$ is a non-negative integer.
%EF9 change qc pour \tilde{c}
    Then the number $\tc_{\bm{\lambda}}$ of $3$-quasi-constellations
    of multitype 
    $(\lambda^{(1)}, \lambda^{(2)}, \lambda^{(3)})$ is given by the formula
%EFMay ne faudrait il pas remplacer n^3/(n-g0)*2^{n-g_0} par n(n-g0)*2^{n+g0}
%cf le calcul qui suit, notamment lemme 21 ?
    \begin{multline*}
        \tc_{\bm{\lambda}} = \frac{n\cdot 2^n}{2^{2g}a(\lambda^{(1)})
    a(\lambda^{(2)}) a(\lambda^{(3)})} \sum_{g_0 + g_1 +g_2 +g_3=g}
    (n-g_0)\bigg[ (g_0!)^2 \binom{n-\ell_1-2g_1}{g_0} \\
    \cdot\binom{n-\ell_2-2g_2}{g_0} 
     \binom{n-\ell_3-2g_3}{g_0}
    S_{g_1}(\lambda^{(1)}) S_{g_2}(\lambda^{(2)}) S_{g_3}(\lambda^{(3)})\bigg].
    \end{multline*}
    \label{PropQC}
\end{prop}

As in the proof of Proposition \ref{prop:GS}, we prefer to work with {\em labelled}
$3$-quasi-constella\-tions 
(circle vertices of each color are labelled separately,
square vertices are {\em not} labelled).
The multitype of such an object is a triple of compositions with obvious definition.
%EF9 change qcL pour \hat{c}
We denote by $\hc(\II,\JJ,\LL)$ the number of labelled 
$3$-quasi-constellations of multitype $(\II,\JJ,\LL)$,
where $\II$, $\JJ$ and $\LL$ are three compositions of the same size $n$.
From now on, $\ell_1$, $\ell_2$ and $\ell_3$ will denote the respective lengths
of the compositions $\II$, $\JJ$ and $\LL$.
Besides, we use the standard index notation $I_i$, $J_i$ and $L_i$
for the $i$-th component of these compositions.

%EF9 reformule pour se passer de la notation cL
In genus $0$ and size $n$, there are clearly $2^n$ times more $3$-quasi-constellations than $3$-constellations 
(indeed in the planar case, cyclically reordering the neighbours around a vertex keeps the property 
of being unicellular; this does not hold in higher genus). Hence, using Equation~\eqref{Eq3ConstPlanar}, we obtain in that case
\begin{equation}\label{EqQCPlanar}
    \hc(\II,\JJ,\LL)=%2^n cL(\II,\JJ,\LL) =
    2^n n^2 (\ell_1-1)! (\ell_2-1)! (\ell_3-1)!.
\end{equation}
%Here, $cL(\II,\JJ,\LL)$ is the number of labelled constellations of multitype
%$(\II,\JJ,\LL)$. The first equality comes from the fact that, .
%Hence relaxing condition \ref{ItemCyclic} in the definition of constellations
%simply lead to a factor $2^n$.
%The second equality follows from .

If we apply our main bijection to a 3-(quasi-)constellation, the tree object is
not necessarily a planar 3-quasi-constellation.
Indeed, one can get square vertices of degree $1$.
Therefore, we need to introduce the concept of prickly planar $3$-quasi-constellations.
\begin{defin}
    A prickly planar $3$-quasi-constellation is a rooted tree with
    circle and square vertices such that:
    \begin{enumerate}[label=(\roman*)]
        \item the circle vertices are colored with $3$ colors $1,2,3$;
        \item all edges have one circle and one square extremity;
        \item each square vertex is either a leaf or 
            linked to exactly one circle vertex of each color;
        \item the number of square leaves linked to vertices of color $1$,$2$, $3$ are the same
            (this number $g_0$ will be called the {\em prickling number}).
    \end{enumerate}
    The size $n$ of such objects is one third of the number of edges.
    For a labelled object, its multitype $(\II,\JJ,\LL)$ is defined as for (quasi-)constellations.

    Planarity is equivalent to the relation:
    \[\ell_1+\ell_2+\ell_3 = 2n-2g_0+1 \] 
\end{defin}
%\alert{Faire un dessin\dots Je laisse ça aux spécialistes d'IPE :). }
%GC8: voici un dessin. Il va ?
\begin{figure}[h!!!!!!]
\includegraphics[scale=0.9]{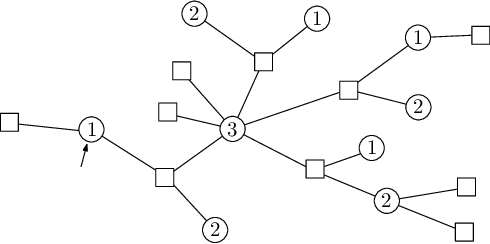}
  \caption{A prickly planar  $3$-quasi-constellation of size $n=6$, with
$l_1=4$, $l_2=4$, $l_3=1$. The prickling number is $g_0=2$.}
\label{fig:pricklyPlanarConstellation}
\end{figure}

\begin{lem}
    The number of labelled prickly planar $3$-quasi-constellations
    of multitype $(\II,\JJ,\LL)$ and prickling number $g_0$ is
    \[ n(n-g_0) 2^{n-g_0} \binom{n-\ell_1}{g_0} \binom{n-\ell_2}{g_0}
    \binom{n-\ell_3}{g_0}
    (\ell_1-1)! (\ell_2-1)! (\ell_3-1)!\] \label{LemPricklyConst}
\end{lem}
\begin{proof}
    It is easier in the proof of this lemma to work with {\em unrooted}
    labelled quasi-constellations.
    Note that a planar (vertex-)labelled tree has no symmetry, hence, dealing
    with unrooted or rooted objects only changes the counting coefficients  by constant
    explicit factors (this would not be true for higher genus objects!).

    Consider an (unrooted) prickly planar $3$-quasi-constellation
    of multitype $(\II,\JJ,\LL)$ and prickling number $g_0$.
    We denote by $a_i$ (resp. $b_i$ and $c_i$) the number of square leaves attached
    to the circle vertex labelled $i$ of color $1$ (resp $2$ and $3$).
    If we erase these leaves we get an unrooted planar $3$-quasi-constellation
    of size $n-g_0$ and multitype $(\II-\aa,\JJ-\bb,\LL-\cc)$,
    where, by definition, $\II-\aa=(I_1-a_1,\dots,I_{\ell_1}-a_{\ell_1})$
    and similar definitions hold for $\JJ-\bb$ and $\LL-\cc$.

    The number of unrooted planar $3$-quasi-constellation
    of multitype $(\II-\aa,\JJ-\bb,\LL-\cc)$ is, by equation 
    \eqref{EqQCPlanar} (beware of the rooting, which yields a factor $n-g_0$):
    \[2^{n-g_0} (n-g_0) (\ell_1-1)! (\ell_2-1)! (\ell_3-1)!. \]
    If we want to recover the prickly planar $3$-quasi-constellation,
    one has to remember for each circle vertex of color $1$ (resp. $2$, $3$)
    where to add the $a_i$ (resp. $b_i$, $c_i$) square leaves.
    This gives 
    \[\prod_{i=1}^{\ell_1} \binom{I_i-1}{a_i} 
    \prod_{i=1}^{\ell_2} \binom{J_i-1}{b_i}
    \prod_{i=1}^{\ell_3} \binom{K_i-1}{c_i} \]
    choices (working with unrooted objects is crucial here).
    Finally the number of unrooted prickly planar $3$-quasi-constellation
    of multitype $(\II,\JJ,\LL)$ and prickling number $g_0$ is
    \begin{multline*}
        2^{n-g_0} (n-g_0) (\ell_1-1)! (\ell_2-1)! (\ell_3-1)!\\
        \cdot \! \sum_{a_1, \dots, a_{\ell_1} \atop
        a_1+\dots+a_{\ell_1}=g_0} \! \prod_{i=1}^{\ell_1} 
        \binom{I_i-1}{a_i} 
        \cdot \! \sum_{b_1,\dots,b_{\ell_2} \atop b_1+\dots+b_{\ell_2}=g_0}
        \! \prod_{i=1}^{\ell_2} 
        \binom{J_i-1}{b_i} 
        \cdot \! \sum_{c_1,\dots,c_{\ell_3} \atop c_1+\dots+c_{\ell_3}=g_0}
        \! \prod_{i=1}^{\ell_3} 
        \binom{L_i-1}{c_i}.
    \end{multline*}
    The first (resp. second, third) sum corresponds to the number of ways
    of choosing $g_0$ elements among $n-\ell_1$ (resp. $n-\ell_2$, $n-\ell_3$).
    This yields the formula of the lemma, because every unrooted object
    can be rooted in $n$ different ways.
\end{proof}

The end of the proof of Proposition \ref{PropQC} 
is now very similar to the one  
of the Goupil-Schaeffer formula (Proposition \ref{prop:GS}).
Therefore, we do not give all the details.

\begin{proof}
    [Sketch of proof of Proposition \ref{PropQC}]
    Our main bijection sends a $3$-quasi-constellation 
    of multitype $(\II,\JJ,\LL)$ to a C-deco\-rated
prickly planar $3$-quasi-constellation of multitype $(\HH,\KK,\MM)$,
where $\HH$, $\KK$ and $\MM$ are refinements of $\II$, $\JJ$ and $\LL$.

For given refinements $\HH$, $\KK$ and $\MM$ of respective lengths
$\ell_1 + 2g_1$, $\ell_2+2g_2$ and $\ell_3+2g_3$,
the number of labelled prickly planar $3$-quasi-constellations 
is (by Lemma \ref{LemPricklyConst})
\begin{multline*}
    n(n-g_0) 2^{n-g_0} \binom{n-\ell_1-2g_1}{g_0} \binom{n-\ell_2-2g_2}{g_0}
    \binom{n-\ell_3-2g_3}{g_0} \\
    \cdot (\ell_1+2g_1-1)! (\ell_2+2g_2-1)! (\ell_3+2g_3-1)!,
 \end{multline*}
    where $g_0 = 1/2 \cdot (2n+1-\ell_1-2g_1-\ell_2-2g_2-\ell_3-2g_3)$
    should be a non-negative integer and is the prickling
    number of the object.

We need the number of refinements $\HH$ of $\II$ of length $\ell_1 +2g_1$,
where each part $I_r$ of $\II$ corresponds to an odd number $2p_r+1$ of parts of $\HH$.
It is given by %EFMay remplace 2p_r+1 par 2p_r
\[\sum_{p_1+\dots+p_{\ell_1}=g_1} \prod_{r=1}^{\ell_1}
\binom{I_r-1}{2p_r}\]
and similar formulas hold for the number of refinements $\KK$ and $\MM$ of $\JJ$ and $\LL$.

When we apply our bijection to a quasi-constellation, we get a C-decorated tree.
As in the proof of Proposition \ref{prop:GS},
to transform this tree into a labelled prickly $3$-quasi-constellations, we need:
\begin{itemize}
    \item to choose for each cycle of circle vertices one distinguished vertex
        (factor $\prod_r 2p_r+1$ for vertices of color $1$ and similar factors
        for other colors);
%EFMay plutot ungroup que group, non?
%    \item to group the square leaves by triples and to choose an orientation
%        for each triple (factor $(g_0!)^2\cdot 2^{g_0}$);
\item to ungroup the $g_0$ cycles of three square leaves
(factor $1/((g_0!)^2\cdot 2^{g_0})$);   
    \item to forget the signs (factor $1/2^{\ell_1+\ell_2+\ell_3+n}$).
\end{itemize}
Finally, we get, by Theorem \ref{theo:main}, that $2^{3n+1} \hc(\II,\JJ,\LL)$
is %EFMay le \prod_r 2p_r+1} est mal placé je crois, j'ai changé
%\begin{multline*}
%\sum_{g_0+g_1+g_2+g_3=g} \frac{2^{\ell_1+\ell_2+\ell_3+n}
%}{(g_0!)^2\cdot 2^{g_0} \prod_r 2p_r+1} \\
%\cdot \sum_{{p_1+\dots+p_{\ell_1}=g_1 \atop q_1+\dots+q_{\ell_2}=g_2}
%\atop s_1+\dots+s_{\ell_3}=g_3} 
%\prod_{r=1}^{\ell_1} \binom{I_r-1}{2p_r+1}  
%\cdot  \prod_{r=1}^{\ell_2} \binom{J_r-1}{2q_r+1}  
%\cdot \prod_{r=1}^{\ell_3} \binom{L_r-1}{2s_r+1}  \\
%\cdot \bigg[ n(n-g_0) 2^n \binom{n-\ell_1-2g_1}{g_0}
%\binom{n-\ell_2-2g_2}{g_0}\\
%    \binom{n-\ell_3-2g_3}{g_0}
%        (\ell_1+2g_1-1)! (\ell_2+2g_2-1)! (\ell_3+2g_3-1)! \bigg]
%    \end{multline*}
%    Indeed, the term in the bracket is the number of labelled
%    prickly planar $3$-quasi-constellations
%    of a given multitype $(\HH,\KK,\MM)$, 
%    the factor in the second line the number of possible multitypes
%    and the first line the ratio between numbers of labelled and $C$-%decorated objects.
\begin{multline*}
\sum_{g_0+g_1+g_2+g_3=g} 2^{\ell_1+\ell_2+\ell_3+n}
(g_0!)^2\cdot 2^{g_0} \\
\!\!\cdot \sum_{{p_1+\dots+p_{\ell_1}=g_1 \atop q_1+\dots+q_{\ell_2}=g_2}
\atop s_1+\dots+s_{\ell_3}=g_3} 
\prod_{r=1}^{\ell_1} \frac1{2p_r+1}\binom{I_r-1}{2p_r}  
\cdot  \prod_{r=1}^{\ell_2} \frac1{2q_r+1}\binom{J_r-1}{2q_r}  
\cdot \prod_{r=1}^{\ell_3} \frac1{2s_r+1}\binom{L_r-1}{2s_r}  \\
\cdot \bigg[ n(n-g_0) 2^{n-g_0} \binom{n-\ell_1-2g_1}{g_0}
\binom{n-\ell_2-2g_2}{g_0}\\
    \binom{n-\ell_3-2g_3}{g_0}
        (\ell_1+2g_1-1)! (\ell_2+2g_2-1)! (\ell_3+2g_3-1)! \bigg],
    \end{multline*}
which equals (using $\ell_1+\ell_2+\ell_3=2n-2g+1$)
\begin{multline*}
\sum_{g_0+g_1+g_2+g_3=g} n(n-g_0)2^{4n-2g+1}
(g_0!)^2 \\
\cdot \binom{n-\ell_1-2g_1}{g_0}S_{g_1}(\lambda^{(1)})
\cdot\binom{n-\ell_2-2g_2}{g_0}S_{g_2}(\lambda^{(2)})
    \cdot\binom{n-\ell_3-2g_3}{g_0}S_{g_3}(\lambda^{(3)}).
    \end{multline*}

    Using the fact that 
    \[\hc(\II,\JJ,\LL)= a(\lambda^{(1)}) a(\lambda^{(2)})
    a(\lambda^{(3)}) \cdot \tc_{\lambda^{(1)},\lambda^{(2)},\lambda^{(3)}}\]
    if $\lambda^{(1)}$, $\lambda^{(2)}$ and $\lambda^{(3)}$ are the sorted
    version of $\II$, $\JJ$ and $\KK$, this ends the proof of 
    Proposition~\ref{PropQC}.
\end{proof}

\subsection{Enumeration of constellations taking only the numbers of vertices into account}
In the previous section, we have seen that Theorem \ref{theo:main} is not suitable
to count unicellular constellations,
as it does not give any information on the rotation system of the map.

Nevertheless, it is possible to get some enumerative results on constellations not
by using Theorem \ref{theo:main} directly, but by mimicking its proof:
find a combinatorial induction for constellations and
then find simpler objects with the same induction.

In this paragraph, we give a combinatorial proof of the enumeration of $3$-constella\-tions
with respect to their number of circle vertices of each color.

\subsubsection{Combinatorial induction for $3$-constellations}

Let us first introduce some notation:
define $\CCC_{\ell_1,\ell_2,\ell_3}(n)$ as the set of $3$-constellations with $\ell_1$
(resp. $\ell_2$ and $\ell_3$) circle vertices of color $1$ (resp. $2$ and $3$).
%    \item[$\CCC^{(m_1,m_2,m_3)}_{(\ell_1,\ell_2,\ell_3)}(n)$] set of the same objects with
%        subsets of size $m_1$, $m_2$ and $m_3$ of marked circle vertices of
%        color $1$, $2$ and $3$.
Its cardinality is denoted $c_{\ell_1,\ell_2,\ell_3}(n)$, and the corresponding genus $g$ is given by $\ell_1+\ell_2+\ell_3=2n-2g+1$.
 For $g=0$ it is well known that $c_{\ell_1,\ell_2,\ell_3}(n)$ 
is the multitype ($3$ types here) Narayana number:
\begin{equation}\label{eq:cg0}
c_{\ell_1,\ell_2,\ell_3}(n)=\frac{1}{n}\binom{n}{\ell_1}\binom{n}{\ell_2}\binom{n}{\ell_3},
\end{equation}
see for instance~\cite{BernardiMoralesSymet} for a combinatorial proof. 

Let us apply Chapuy's bijection (recalled in Subsection~\ref{SubsectTrisections}) to a 
$3$-constellation of genus $g$ in $\CCC_{\ell_1,\ell_2,\ell_3}(n)$.
Two things may happen:
\begin{itemize}
    \item either the sliced vertex is a circle vertex, in which case we get
        a $3$-constella\-tion of genus $g-h$ with $2h+1$ marked circle vertices
        of the corresponding color (for some $h \ge 1$);
    \item or the sliced vertex is a square vertex.
        In this case, as square vertices have degree $3$, the resulting map
        has always genus $g-1$ and has three square leaves.
        Erasing these leaves, we get a $3$-constellation of size $n-1$
        and genus $g-1$.
\end{itemize}
% GC8: mis partout "mapping" plutôt que "map", qui est un peu confusionnant...
In the first case, Chapuy's inverse mapping works well and always produces a constellation
in $\CCC_{\ell_1,\ell_2,\ell_3}(n)$.

In the second case, one has to choose where to add square leaves, that is
to choose a corner $c_i$ of
a circle vertex of color $i$ for $i=1,2,3$ ($n(n-1)^2$ choices;
beware of the root).
Then, when we apply Chapuy's inverse mapping, it may happen that the newly created
square vertex does not fulfill condition~\ref{ItemCyclic} of the definition
of constellations (Definition~\ref{def:constellation}).
\begin{lem}
    Chapuy's inverse mapping leads to a constellation if and only if
    the three chosen corners appear in the cyclic order $(c_1,c_2,c_3)$
    when we turn 
% GC8: added "clockwise" (comme dans la description de ma bijection au début).
clockwise
around the unique face of the map.
    \label{LemGoodTriplet}
\end{lem}
\begin{proof}
%    See the details of the inverse construction \cite[paragraph ??]{Ch09}.
%    \alert{(Guillaume, je te laisse mettre la ref vers le bon paragraphe de ton papier
%    et plus de détails si tu le juges nécessaire.)}
% GC8: cela ne me semble même pas nécessaire, la construction étant entièrement
% décrite au paragraphe 2.2. (dans l'autre sens, certes, mais on s'en fiche
% non? puisque c'est une bijection...)
This is direct by construction: let $f,f',f''$ be three leaves in a unicellular
map and $e,e',e''$ the edges incident to them, respectively. Then $(e,e',e'')$
appear in counterclockwise order around the new vertex created by the gluing of the
three leaves if and only if $f,f',f''$ appear in clockwise order around the face in
the original map. 
See Subsection~\ref{SubsectTrisections} (or \cite[Paragraph 4.2]{Ch09} for a
direct description of the inverse mapping). 
\end{proof}
% GC8: remplacé partout "triplet" par "triple"
\begin{lem}
    For each constellation of size $n-1$ and genus $g-1$,
    exactly $n^2(n-1)/2$ triples of corners $(c_1,c_2,c_3)$
    over $n(n-1)^2$ satisfy the condition above.
\end{lem}
\begin{proof}
%    Left to the reader.
% GC8: un peu vache! J'ai donné une indication
Observe that a rooted unicellular $3$-constellation of size $n-1$ has
$n$ (resp. $n-1$) corners incident to circle vertices 
of color $1$ (resp. $2$ and $3$),
and that clockwise around the face the colors of these
$3n-2)$ corners appear in the order $(1,2,3,1,2,3,\dots, 1,2,3,1)$. From there
it is a simple exercise to check the statement of the lemma.
\end{proof} 

Finally, with $g$ defined by $\ell_1+\ell_2+\ell_3+2g=2n+1$, we get the following inductive relation (for $g>0$):
\begin{multline}
    2g\ \! c_{\ell_1,\ell_2,\ell_3}(n) =\frac{n^2(n-1)}{2} c_{\ell_1,\ell_2,\ell_3}(n-1)
    +\sum_{h \ge 1}
    \binom{\ell_1+2h}{2h+1} c_{\ell_1+2h,\ell_2,\ell_3}(n)\\ +
    \binom{\ell_2+2h}{2h+1} c_{\ell_1,\ell_2+2h,\ell_3}(n)+
    \binom{\ell_3+2h}{2h+1} c_{\ell_1,\ell_2,\ell_3+2h}(n).
    \label{EqIndConst}
\end{multline}
%EFMay ajout de +1 dans le membre de droite, non?
Note that the coefficients $c_{\ell_1,\ell_2,\ell_3}(n)$ are completely specified by the induction and by~\eqref{eq:cg0} (expression for $g=0$). 

We can define some refinement of the number $c_{\ell_1,\ell_2,\ell_3}(n)$ by the induction (with $g$ defined as usual by $\ell_1+\ell_2+\ell_3+2g=2n+1$)
\begin{multline*}
    2g\ c_{\ell_1,\ell_2,\ell_3}(n;g_0) =\frac{n^2(n-1)}{2} c_{\ell_1,\ell_2,\ell_3}(n-1;g_0-1)
    +\sum_{h \ge 1}                                                                       
        \binom{\ell_1+2h}{2h+1} c_{\ell_1+2h,\ell_2,\ell_3}(n;g_0)\\+                                  
            \binom{\ell_2+2h}{2h+1} c_{\ell_1,\ell_2+2h,\ell_3}(n;g_0)+                                  
                \binom{\ell_3+2h}{2h+1} c_{\ell_1,\ell_2,\ell_3+2h}(n;g_0)
\end{multline*}
and initial conditions that for $g=0$, $c_{\ell_1,\ell_2,\ell_3}(n;0)=c_{\ell_1,\ell_2,\ell_3}(n)$ and $c_{\ell_1,\ell_2,\ell_3}(n;g_0)=0$ for $g_0>0$. 
%EFMay je crois que ca ne suffit pas, il faut tout de meme definir le cas g=0% cf aussi remarque au Corollaire 25
%\begin{align*}
%    c_{\ell_1,\ell_2,\ell_3}(n;g_0) &= 0\text{ for }g>0
%    \ell_1+\ell_2+\ell_3+2g_0 > 2n+1 ;\\
%    c_{1,1,1}(1;0)=1.
%\end{align*}

Then an immediate induction on $g$ proves that
\[c_{\ell_1,\ell_2,\ell_3}(n)= \sum_{g_0 \ge 0}  c_{\ell_1,\ell_2,\ell_3}(n;g_0).\]

Note that the parameter $g_0$ does not {\em a priori} have a combinatorial
interpretation.
It is a computational artefact introduced to mimic an induction
relation on planar objects that we shall see in the next section.

\subsubsection{A planar object with (almost) the same induction}
Denote by $d_{\ell_1,\ell_2,\ell_3}(n;g_0)$ the number of planar $3$-constellations of
size $n-g_0$ endowed with:
\begin{itemize}
    \item a $C$-permutation of its circle vertices of each color with
        respectively $\ell_1$, $\ell_2$ and $\ell_3$ cycles;
    \item 
%EF9 $g_0$ indistinguishable triples of $(f_1,f_2,f_3)$ of square leaves
     %   where $f_i$ is attached to a circle vertex of color $i$.
An unordered set of $g_0$ triples of square leaves, such that the triples are mutually disjoint
and each triple is of the form $f_1,f_2,f_3$, with $f_i$ connected to a circle vertex of color $i$.  
\end{itemize}
\begin{lem}
    These numbers satisfy the induction (with $g$ defined as usual by $\ell_1+\ell_2+\ell_3=2n-2g+1$), for $g>0$:
\begin{multline*}
    2g\ d_{\ell_1,\ell_2,\ell_3}(n;g_0) =2 n(n-1)^2 d_{\ell_1,\ell_2,\ell_3}(n-1;g_0-1)
    +\sum_{h \ge 1}                                                                       
        \binom{\ell_1+2h}{2h+1} d_{\ell_1+2h,\ell_2,\ell_3}(n;g_0)\\+                                  
            \binom{\ell_2+2h}{2h+1} d_{\ell_1,\ell_2+2h,\ell_3}(n;g_0)+                                  
                \binom{\ell_3+2h}{2h+1} d_{\ell_1,\ell_2,\ell_3+2h}(n;g_0).
\end{multline*}
\end{lem}
\begin{proof}
    First, we have
    \[g_0\ \! d_{\ell_1,\ell_2,\ell_3}(n;g_0) = n (n-1)^2 d_{\ell_1,\ell_2,\ell_3}(n-1;g_0-1).\]
    Indeed, the left-hand side counts the same objects as above
    with a marked triple of square leaves.
    If we erase this triple, we get objects
    counted by $d_{\ell_1,\ell_2,\ell_3}(n-1;g_0-1)$.
    This can be inverted if we remember in which corners the leaves were attached  
    ($n(n-1)^2$ possibilities ; beware of the root).

    Second, the induction for $C$-permutations leads to 
    \begin{multline*}
        2(g-g_0) d_{\ell_1,\ell_2,\ell_3}(n;g_0) = \sum_{h \ge 1}
        \binom{\ell_1+2h}{2h+1} d_{\ell_1+2h,\ell_2,\ell_3}(n;g_0)\\+                                  
            \binom{\ell_2+2h}{2h+1} d_{\ell_1,\ell_2+2h,\ell_3}(n;g_0)+                                  
                \binom{\ell_3+2h}{2h+1} d_{\ell_1,\ell_2,\ell_3+2h}(n;g_0).
            \end{multline*}
\end{proof}

\begin{corol}\label{coro:cd}
    For each $n,g_0,\ell_1,\ell_2,\ell_3$, we have:
%EFMay je crois qu'il faut multiplier le membre de gauche par
%2^{2n-2g0+1} pour qu'il y ait une egalite pour g=0
    \[2^{2n+1} (n-g_0) \, c_{\ell_1,\ell_2,\ell_3}(n;g_0) = n \, d_{\ell_1,\ell_2,\ell_3}(n;g_0).\]
\end{corol}
\begin{proof}
%EFMay il faudrait detailler un peu plus, initial conditions devrait
%correspondre a g=0
Denote by $u_{\ell_1,\ell_2,\ell_3}(n;g_0)$ the left-hand side and by $v_{\ell_1,\ell_2,\ell_3}(n;g_0)$ the right-hand side.
For $g=0$ and $g_0=0$ the left-hand side equals the right-hand side (the factor $2^{2n+1}$ is due to the signs of the cycles ---a cycle of length $1$ on each of the
$2n+1$ vertices--- for the objects on the right-hand side).
For $g=0$ and $g_0>0$ both sides are zero. For $g>0$, we have the inductive relation (inherited from the inductive relation for $c_{\ell_1,\ell_2,\ell_3}(n;g_0)$)
\begin{multline*}
    2g\ \! u_{\ell_1,\ell_2,\ell_3}(n;g_0) =2n^2(n-1) u_{\ell_1,\ell_2,\ell_3}(n-1;g_0-1)
    +\sum_{h \ge 1}                                                                       
        \binom{\ell_1+2h}{2h+1} u_{\ell_1+2h,\ell_2,\ell_3}(n;g_0)\\+                                  
            \binom{\ell_2+2h}{2h+1} u_{\ell_1,\ell_2+2h,\ell_3}(n;g_0)+                                  
                \binom{\ell_3+2h}{2h+1} u_{\ell_1,\ell_2,\ell_3+2h}(n;g_0),
\end{multline*}
and we can see that $v_{\ell_1,\ell_2,\ell_3}(n;g_0)$ satisfies exactly the same relation (inherited from the inductive relation for $d_{\ell_1,\ell_2,\ell_3}(n;g_0)$).
Hence the left-hand side equals the right-hand side for all values of $\ell_1,\ell_2,\ell_3,n,g_0$.
\end{proof}

\subsubsection{Final computation}
The nice feature in this statement is that, as it counts planar objects,
$d_{\ell_1,\ell_2,\ell_3}(n;g_0)$ can be easily computed combinatorially.
As usual, denote $g=1/2 (2n+1-\ell_1-\ell_2-\ell_3)$.

\begin{lem}
    For any integers $\ell_1,\ell_2,\ell_3 \ge 1$ and $g_0 \ge 0$, one has
    \[
    d_{\ell_1,\ell_2,\ell_3}(n;g_0) = \frac{2^{\ell_1+\ell_2+\ell_3} \, (n!)^3\, (n-g_0)}{
        n^2 \, g_0!\, \ell_1!\, \ell_2!\, \ell_3!} 
        \cdot \sum_{g_1+g_2+g_3=g-g_0} \, 
        \prod_{i=1}^3 \,\frac{P_{g_i}(\ell_i)}{(n-\ell_i -g_0-2g_i)!},
\]
where \[P_{h}(x)=\sum_{\gamma \vdash h} \frac{(x)_{\ell(\gamma)}}
{\prod_i m_i(\gamma)! (2i+1)^{m_i(\gamma)}}.\]
    \label{lem:computing_d}
\end{lem}
\begin{proof}
    Fix $g_1,g_2,g_3 \ge 0$ with $g_1+g_2+g_3=g-g_0$.
%EFMay raccourcit !    
%    Let us first count the number $c_{(r_1,r_2,r_3)}(n-g_0)$ 
%    of planar 3-constellations
%    with $r_1:=\ell_1+2g_1$ (resp. $r_2:=\ell_2+2g_2$ and $r_3:=\ell_3+2g_3$)
%    circle vertices of color $1$ (resp. $2$ and $3$).
%    One has to sum over the possible multitypes
%    $(\lambda^{(1)},\lambda^{(2)},\lambda^{(3)})$ and use 
%    formula~\eqref{Eq3ConstPlanar}:
%    \begin{align*}
%        c_{(r_1,r_2,r_3)}(n-g_0) &=
%    \sum_{\lambda^{(1)} \vdash n-g_0 \atop \ell(\lambda^{(1)})=r_1 }
%    \sum_{\lambda^{(2)} \vdash n-g_0 \atop \ell(\lambda^{(2)})=r_2 }
%    \sum_{\lambda^{(3)} \vdash n-g_0 \atop \ell(\lambda^{(3)})=r_3 }
%    \frac{(n-g_0)^2 (r_1-1)! (r_2-1)! (r_3-1)!}
%    {a(\lambda^{(1)}) a(\lambda^{(2)}) a(\lambda^{(3)})}\\
%    &=\frac{(n-g_0)^2}{r_1 r_2 r_3} 
%        \binom{n-g_0-1}{r_1-1}
%        \binom{n-g_0-1}{r_2-1}
%        \binom{n-g_0-1}{r_3-1}.
%\end{align*}
%We used three times the equation
%\[\sum_{\lambda \vdash n, \ell(\lambda)=s} \frac{s!}{a(\lambda)}=\binom{n-1}{s-1}\]
%that counts compositions of $n$ of length $s$ in two different ways.  
The number $c_{r_1,r_2,r_3}(n-g_0)$ 
    of planar 3-constellations
    with $r_1:=\ell_1+2g_1$ (resp. $r_2:=\ell_2+2g_2$ and $r_3:=\ell_3+2g_3$)
    circle vertices of color $1$ (resp. $2$ and $3$) is given
by the multitype ($3$ colors here) Narayana number
$$
\frac1{n-g_0}\binom{n-g_0}{r_1}\binom{n-g_0}{r_2}\binom{n-g_0}{r_3}.
$$
Using the same arguments as in Section \ref{SubsectImmediateCorol},
we see that the number of possible choices for the $C$-permutation of the
circle vertices of such a constellation with $\ell_1$ (resp. $\ell_2$
and $\ell_3$) cycles of vertices of color $1$ (resp. $2$ and $3$) is
\[\prod_{i=1}^3 2^{\ell_i} \frac{(r_i)!}{(\ell_i)!} P_{g_i}(\ell_i). \]
Finally, one has $(n-g_0)\dots(n-1)$ ways to add one by one
$g_0$ square leaves
connected to vertices of color $2$ (resp. $3$).
Because of the root, there are $(n-g_0+1)\dots(n)$ to do it for vertices
of color $1$.
As we added these leaves one by one, we can pack them into triples
in a canonical way.
But these triplets are ordered, so we shall divide by $g_0!$ at the end.

Putting everything together, we get the formula stated in the lemma.
\end{proof}

Finally, from Corollary~\ref{coro:cd} and Lemma~\ref{lem:computing_d}, we get a combinatorial proof of the following
enumeration formula for $3$-constellations with respect to the
number of circle vertices of each color.
\begin{prop}
    For any $n,\ell_1,\ell_2,\ell_3 \ge 1$, with $g$ defined by $\ell_1+\ell_2+\ell_3=2n-2g+1$, one has:
%EFMay remplacer le denom 2^{2g0} par 2^{2n+1} ?
    \[c_{\ell_1,\ell_2,\ell_3}(n) =\frac{n!^2\ \!(n\!-\!1)!}{
        2^{2g}\ \ell_1!\, \ell_2!\, \ell_3!} 
        \cdot \sum_{g_0+g_1+g_2+g_3=g} \frac{1}{g_0!}\, 
        \prod_{i=1}^3 \,\frac{P_{g_i}(\ell_i)}{(n-\ell_i -g_0-2g_i)!},\]
    where \[P_{h}(x)=\sum_{\gamma \vdash h} \frac{(x)_{\ell(\gamma)}}
{\prod_i m_i(\gamma)! (2i+1)^{m_i(\gamma)}}.\]
\end{prop}

This formula can alternatively be deduced from the Poulalhon Schaeffer formula
 using~\cite[Lemma 4.1]{GoSc00}.
If we could refine~\eqref{EqIndConst} so as to control the degree distribution,
we could give a purely combinatorial proof of the Poulalhon-Schaeffer formula in the case
$m=3$.

\subsection{Conclusion on constellations}

We are unfortunately not able to give a combinatorial proof of the Poulalhon-Schaeffer formula
(even in the case $m=3$).
Nevertheless, the work presented here suggests that the different elements in this
formula have a combinatorial meaning.
The case $m>3$ seems even harder.

\section*{Acknowledgements}
We are indebted to an anonymous referee for suggesting the
refinement of the Harer-Zagier summation formula presented in Subsection
\ref{subsec:Refinement_Harer_Zagier}.

\bibliographystyle{abbrv}
\bibliography{bib_unicellu}

\end{document}